%% file: GarRuz4.tex
\documentclass[12pt]{amsart}
\usepackage{latexsym}
\usepackage{amssymb} 
\usepackage{mathrsfs}
\usepackage{amsmath}
\usepackage{latexsym}
\usepackage{delarray}
\usepackage{amssymb,amsmath,amsfonts,amsthm,
mathrsfs}

\usepackage{hyperref}
\renewcommand\eqref[1]{(\ref{#1})} 

\input{mymacro}

\newcommand{\espo}{\mathrm{e}}
\newcommand\Rn{{\mathbb R}^n}

\renewcommand\N{{\mathbb N}_0}

\title[Hyperbolic equations with non-regular coefficients]{
Hyperbolic  second order equations with non-regular time dependent coefficients}

\author[Claudia Garetto]{Claudia Garetto}
\address{
  Claudia Garetto:
  \endgraf
  Department of Mathematical Sciences
  \endgraf
  Loughborough University
  \endgraf
  Loughborough, Leicestershire, LE11 3TU
  \endgraf
  United Kingdom
  \endgraf
  {\it E-mail address} {\rm c.garetto@lboro.ac.uk}
  }
\author[Michael Ruzhansky]{Michael Ruzhansky}
\address{
  Michael Ruzhansky:
  \endgraf
  Department of Mathematics
  \endgraf
  Imperial College London
  \endgraf
  180 Queen's Gate, London SW7 2AZ
  \endgraf
  United Kingdom
  \endgraf
  {\it E-mail address} {\rm m.ruzhansky@imperial.ac.uk}
  }

\thanks{The second author was supported by the
EPSRC Leadership Fellowship EP/G007233/1.
}
\date{}

\subjclass[2010]{Primary 35L10; 35D30; Secondary 46F05;}
\keywords{Hyperbolic equations, Gevrey spaces, ultradistributions, weak solutions}

\begin{document}

\maketitle

\begin{abstract}
In this paper we study weakly hyperbolic second order equations with time dependent irregular coefficients.  
This means to assume that the coefficients are less regular than H\"older. The 
characteristic roots are also allowed to have multiplicities.
For such equations, we describe the notion of a `very weak solution' adapted to the type of solutions that exist for
regular coefficients. The construction is based on considering Friedrichs-type mollifiers of coefficients and
corresponding classical solutions, and
their quantitative behaviour in the regularising parameter.
We show that even for distributional coefficients, the Cauchy problem does have a very weak solution,
and that this notion leads to classical or ultradistributional solutions under conditions when
such solutions also exist. In concrete applications, the dependence on the regularising parameter 
can be traced explicitly. 
\end{abstract}

\section{Introduction}
In this paper we study equations of the type
\beq
\label{intro_eq}
\partial_t^2u(t,x)-\sum_{i=1}^n b_i(t)\partial_t \partial_{x_i}u(t,x)-\sum_{i=1}^n a_i(t)
\partial_{x_i}^2u(t,x)=0,\qquad x\in\R^n,\, t\in[0,T],
\eeq
where the coefficients are real and $a_i\ge 0$ for all $i=1,\dots,n$. It follows that this equation is hyperbolic 
(but not necessarily strictly hyperbolic). 
This kind of equations appears in many physical phenomena where discontinuous or singular entities are involved, 
for instance in the wave propagation in a layered medium. An example is the wave equation
\[
\partial_t^2u(t,x)-\sum_{i=1}^n a_i(t)\partial_{x_i}^2u(t,x)=0,
\]
where the coefficients $a_i$ are Heaviside or Delta functions. In particular, when $n=2$, we can have
the equation
\begin{equation}\label{EQ:wave-example}
\partial_t^2u(t,x)-H_{t_0, t_1}(t)\partial_{x_1}^2u(t,x)-\delta_{t_2}\partial_{x_2}^2u(t,x)=0,
\end{equation}
where $0<t_0< t_1\le t_2\le T$, $H_{t_0,t_1}$ is the jump function with $H_{t_0,t_1}(t)=0$  for $t<t_0$ and $t>t_1$ and $H_{t_0,t_1}(t)=1$ for $t_0\le t\le t_1$, 
and $\delta_{t_2}$ is the Delta function concentrated at $t_2$. In this paper 
we will use the expression 
a real-valued distribution for a distribution $u\in\D(\R)$ such that
$u(\varphi)\in\R$ for all real-valued test functions $\varphi$. 
Similarly, we will write $u\geq 0$ if 
$u(\varphi)\geq 0$ for all non-negative test functions $\varphi\geq 0$.

This is clearly the case of the coefficients above.

As usual, we will often rewrite the equation \eqref{intro_eq} using 
the notation $D_t=-{\rm i}\partial_t$ and $D_{x_i}=-{\rm i}\partial_{x_i}$.
The well-posedness of the corresponding Cauchy problem
\beq
\label{intro_CP}
\begin{split}
D_t^2u(t,x)-\sum_{i=1}^n b_i(t)D_tD_{x_i}u(t,x)-\sum_{i=1}^n a_i(t)D_{x_i}^2u(t,x)&=0,\\
u(0,x)&=g_0,\\
D_t u(0,x)&=g_1,\\
\end{split}
\eeq
has been studied by many authors in the case of regular coefficients. If the coefficients 
$a_i$ and $b_i$ are sufficiently regular we can refer to the fundamental paper by
Colombini, de Giorgi and Spagnolo
\cite{Colombini-deGiordi-Spagnolo-Pisa-1979}, showing that even if the coefficients are smooth,
the well-posedness of the Cauchy problem \eqref{intro_CP} can be expected to hold only in
Gevrey spaces. In fact, the famous example of Colombini and Spagnolo 
\cite{Colombini-Spagnolo:Acta-ex-weakly-hyp} shows that even if all $b_i=0$ and 
all $a_i$ are smooth, the Cauchy problem \eqref{intro_CP} may not be distributionally
well-posed due to multiplicities. On the other hand, if the operator \eqref{intro_eq} is
strictly hyperbolic, it was shown in \cite{Colombini-Jannelli-Spagnolo:Annals-low-reg} that
the Cauchy problem \eqref{intro_CP}  may be still distributionally not well-posed if the
coefficients are less regular, e.g. only H\"older.

These examples, already for the second order equations with time-dependent
coefficients as in \eqref{intro_CP}, show the
following by now well-known qualitative facts:

\begin{itemize}
\item if the coefficients are smooth and the equation is strictly hyperbolic,  
the Cauchy problem \eqref{intro_CP} is distributionally well-posed
(of course, much more is known, but it is less important for our purposes here);
\item if the coefficients are smooth but the equation has multiplicities, then the Cauchy problem 
\eqref{intro_CP} may be not distributionally well-posed. However, it becomes well-posed in the
appropriate classes of ultradistributions (depending on additional properties of coefficients or
characteristic roots);
\item if the equation is strictly hyperbolic but the coefficients are only H\"older continuous,
the Cauchy problem 
\eqref{intro_CP} may be not distributionally well-posed. However, it becomes well-posed in the
appropriate classes of ultradistributions;
\item if the coefficients of the equation are continuous (and not H\"older continuous), there
may be no ultradistributional well-posedness. However, it may become well-posed in the
space of Fourier hyperfunctions.
\end{itemize}
As we see in the above statements, if we want to continue having a well-posedness
result, the reduction in regularity assumptions on the coefficients
leads to the necessity to weaken the notion of solution to the Cauchy problem
and to enlarge the allowed class of solutions. 

A threshold between distributional and ultradistributional well-posedness for equations
with time-dependent coefficients
(on the level of $C^\infty$ and Gevrey well-posedness) in terms of the regularity of
coefficients has been
discussed by Colombini, del Santo and Reissig 
\cite{Colombini-del-Santo-Reissig:BSM-2003}. 
We note that for $x$-dependent coefficients the situation becomes even much more
subtle: for example, while very general Gevrey well-posedness
results are available for Gevrey coefficients
(see, e.g. Bronshtein \cite{Bronshtein:TMMO-1980} or Nishitani \cite{Nishitani:BSM-1983}),
the $C^\infty$ well-posedness of second order equations with smooth coefficients
is heavily dependent on the geometry of characteristics
(see, e.g. 
\cite{Bernardi-Parenti-Parmeggiani:CPDE-2012,Parenti-Parmeggiani:CPDE-2009}).
Again, most of such results can be translated into distributional or ultradistributional
well-posedness, but still for equations with smooth or Gevrey coefficients.

The aim of this paper is to analyse the Cauchy problem \eqref{intro_CP} under much weaker
regularity assumptions on coefficients. The general goal of reducing the regularity of coefficients
for evolution partial differential equations has both mathematical and physical motivations,
and has been thoroughly discussed by Gelfand \cite{Gelfand:analysis-PDEs-UMN-1963}, 
to which we refer also for the philosophical discussion of this topic.

Before we proceed with our approach, let us mention that  
the Cauchy problem \eqref{intro_CP} for operators with irregular coefficients has 
history and motivations from specific applied sciences. 
For example, problems of this type appear in geophysical applications 
with delta-like sources and discontinuous or more irregular media
(for example, fractal-type media occurs naturally in the upper crust of the Earth or in 
fractured rocks), see \cite{Marsan-Bean}, and especially 
\cite{Hormann-de-Hoop:AAM-2001} for a more detailed discussion and further 
references in geophysics and in tomography. 
Such problems have been treated using microlocal constructions in
the Colombeau algebras, see e.g.
H\"ormann and de Hoop \cite{Hormann-de-Hoop:AAM-2001,Hormann-de-Hoop:AA-2002}.
If the coefficients are measurable such equations often fall in the scope of problems
which can be handled by semigroup methods, as in Kato \cite{Kato:bk-1995}.
However, to the best of our knowledge, there are no approaches to problems
with irregularities like those in \eqref{EQ:wave-example}, providing both a well-posedness
statement and a relation to `classical' solutions.

In this paper, we will look at the Cauchy problem \eqref{intro_CP} in different settings, the most
general being that the coefficients $a_i$ and $b_i$ are distributions. In this case, in view of
the famous Schwartz impossibility result on multiplication of distributions
\cite{Schwartz:impossibility-1954}, the first question that already arises is how to interpret the equation
\eqref{intro_eq} when $u$ is a distribution as well. And, a related question for our purposes,
is how to interpret the notion of a solution to the Cauchy problem \eqref{intro_CP}.
In view of the discussion above, it appears natural that in order to obtain solutions in
this setting, one should weaken the notion of a solution to the Cauchy problem since 
ultradistributions or hyperfunctions may not be sufficient for such purpose.
 
The aim and the main results of this paper are to show that 
\begin{itemize}
\item one can introduce the notion of `very weak solutions' to the Cauchy problem
 \eqref{intro_CP}, based on regularising coefficients and the Cauchy data with certain
 adaptation of Friedrichs mollifiers.
 Then, one can show that very weak solutions exist even if the coefficients
 and the Cauchy data are (compactly supported) distributions (Theorem \ref{theo_vws});
\item if the coefficients are sufficiently regular, namely, if they are in the class $C^2$, 
the very weak solutions all coincide in a certain sense, and are related to (coincide with)
other known solutions. More precisely, if the Cauchy data are Gevrey ultradifferentiable
functions, any very weak solution (for any regularisation of the coefficients) converges
in the strong sense
to the classical solution in the limit of the regularisation parameter.
If the Cauchy data are distributions, any very weak solution 
(for any regularisation of the coefficients) converges
in the ultradistributional sense
to the ultradistributional solution in the limit of the regularisation parameter.
See Theorem \ref{theo_vws2} for a precise formulation.
\end{itemize}

The appearance of the class $C^2$ is due to the fact that since we do not assume that the
equation is strictly hyperbolic, the $C^2$-regularity of coefficients does guarantee that the
characteristic roots of \eqref{intro_eq} are Lipschitz, and hence we know that the Gevrey
or ultradistributional well-posedness holds. In the case the equation is strictly hyperbolic,
the assertions above still hold if the coefficients are e.g. Lipschitz. Some further refinements
are possible given precise relations between regularities of coefficients and 
roots of a hyperbolic polynomial
(Bronshtein's theorem \cite{Bronshtein:roots-SMJ-1979} and its refinements as in 
\cite{Colombini-Orru-Pernazza:roots-IJM-2012}). 

The idea of considering regularisations of coefficients or solutions of 
hyperbolic partial differential equations 
in different senses is of course natural. 
For example, after regularising (e.g. non-Lipschitz, H\"older, etc.) coefficients with a parameter $\eps$,
relating $\eps$ to some frequency zones in the energy estimate often yields the Gevrey or even
$C^\infty$ well-posedness (see e.g. 
\cite{Colombini-deGiordi-Spagnolo-Pisa-1979,Colombini-del-Santo-Kinoshita:ASNS-2002}, 
and other papers).
It is not always possible to relate $\eps$ to frequency zones in which case families 
of solutions can be considered as a whole: for example, 
for hyperbolic equations with discontinuous
coefficients regularised families have been already considered by Hurd and Sattinger
\cite{Hurd-Sattinger}, with a subsequent analysis of limits of these regularisations in 
$L^2$ as $\eps\to 0$.

The purpose of this paper is to carry out a thorough analysis of appearing families of solutions
and, by formulating a naturally associated notion of `very weak' solution, to relate it
(as $\eps\to 0$) to known classical, distributional or ultradistributional solutions.

In the next section we provide more specifics to the above statements. In particular, 
we briefly review the relevant ultradistributional well-posedness results, and put the
notion of a very weak solution to a wider context. 

In what concerns the literature review
for second order Cauchy problems \eqref{intro_CP}, we will only give very specific
references relevant to our subsequent purposes: for `regular' coefficients much is known,
for sharp results see e.g. already
Colombini, de Giorgi, Spagnolo \cite{Colombini-deGiordi-Spagnolo-Pisa-1979}, 
Nishitani \cite{Nishitani:BSM-1983}. Also, we do not discuss other interesting 
phenomena on the borderline of the existence of strong solutions
(e.g. irregularity in $t$ can be sometimes
compensated by favourable behaviour in $x$, see e.g.
Cicognani and Colombini \cite{Cicognani-Colombini:JDE-2013}).

\section{Main results}
 
As we mentioned in the introduction, already when the coefficients are regular,
there are several types of assumptions where we can
expect qualitatively different results. On one hand, for very regular data,
we may have well-posedness in the spaces of smooth, Gevrey, or analytic
functions. At the duality level, this corresponds to the well-posedness in
spaces of distributions, ultradistributions, or Fourier hyperfunctions. 
  
We start by recalling the known results for coefficients which are
regular: in \cite{GR:12}, extending the one-dimensional result of 
Kinoshita and Spagnolo in \cite{KS}, 
we have obtained the following well-posedness result:
\begin{theorem}[\cite{GR:12}]\label{THM:GR12}
\leavevmode
\begin{itemize}
\item[(i)]
If the coefficients $a_j$, $b_j$, $j=1,\dots,n$, belong to ${C}^k([0,T])$ for some $k\ge 2$ and
$g_j\in \gamma^s(\R^n)$ for $j=1,2$ then there exists a unique solution $u\in C^2([0,T];\gamma^s(\R^n))$
of the Cauchy problem \eqref{intro_CP} provided that
\[
1\le s<1+\frac{k}{2};
\]
\item[(ii)] if the coefficients are of class $C^\infty$ on $[0,T]$ then the Cauchy problem
\eqref{intro_CP} is well-posed in any Gevrey space;
\item[(iii)] under the hypotheses of (i), if the initial data 
$g_j$ are Gevrey Beurling ultradistributions in $\mathcal E'_{(s)}(\R^n)$ for $j=1,2$ then there exists a unique solution $u\in C^2([0,T];\mathcal{D}'_{(s)}(\R^n))$
of the Cauchy problem \eqref{intro_CP} provided that
\[
1\le s<1+\frac{k}{2};
\]
\item[(iv)] under the hypotheses of (ii) the Cauchy problem \eqref{intro_CP} is well-posed in any space of ultradistributions;
\item[(v)] finally if the coefficients are analytic on $[0,T]$ then the Cauchy problem \eqref{intro_CP} is $C^\infty$ and distributionally well-posed.
\end{itemize}
\end{theorem}
For the sake of the reader we briefly recall the definitions of the spaces
$\gamma^s(\Rn)$ and $\gamma^{(s)}(\Rn)$  of (Roumieu) Gevrey functions and (Beurling) Gevrey functions, respectively. These are intermediate classes between analytic functions ($s=1$) and smooth functions. In the sequel, $\N=\{0,1,2,\dots\}$. 

\begin{definition}\label{def_gevrey}
Let $s\geq 1$.
We say that $f\in C^\infty(\R^n)$ belongs to the Gevrey (Roumieu) 
class $\gamma^s(\R^n)$ if for every compact set $K\subset\R^n$ there
exists a constant $C>0$ such that for all $\alpha\in\N^n$ we have the estimate
\[
\sup_{x\in K}|\partial^\alpha f(x)|\le C^{|\alpha|+1}(\alpha!)^s.
\]
We say that $f\in C^\infty(\R^n)$ belongs to the Gevrey (Beurling) 
class $\gamma^{(s)}(\R^n)$ if for every compact set $K\subset\R^n$ and for every $A>0$ there exists a constant $C>0$ such that for all $\alpha\in\N^n$ we have the estimate
\[
\sup_{x\in K}|\partial^\alpha f(x)|\le CA^{|\alpha|+1}(\alpha!)^s.
\]
\end{definition}
Let now $\gamma^{(s)}_c(\R^n)$ be the space of Beurling Gevrey functions with compact support. Its dual is the corresponding space $\D'_{(s)}(\R^n)$ of  ultradistributions and $\mathcal E'_{(s)}(\R^n)$ is the subspace of compactly supported ultradistributions.  We refer to \cite{GR:11} for relevant properties and Fourier characterisations of these spaces of ultradifferentiable functions and ultradistributions.

Going back to the equation \eqref{intro_eq} and the corresponding Cauchy problem, we know that dropping the regularity assumption on the coefficients from $C^k$ to $C^{2\alpha}$, with $\alpha\in(0,1]$, we still get Gevrey and ultradistributional well-posedness as proved e.g. in \cite{ColKi:02} for $n=1$, and in \cite{GR:11} for general $n$:
\begin{theorem}[\cite{ColKi:02}]
Assume that the characteristic roots of the equation \eqref{intro_eq} are of class $C^\alpha$, $\alpha\in(0,1]$ in $t$.
\begin{itemize}
\item[(i)] Let $g_j\in \gamma^s(\R^n)$ for $j=1,2$. Hence, the Cauchy problem \eqref{intro_CP} 
has a unique solution $u\in C^2([0,T], \gamma^s(\R^n))$ provided that
\[
1\le s<1+\alpha.
\]
\item[(ii)] Let $g_j\in \E'_{(s)}(\R^n)$ for $s=1,2$. Hence, the Cauchy problem \eqref{intro_CP} has a unique solution $u\in C^2([0,T], \D'_{(s)}(\R^n))$ provided that
\[
1\le s\le 1+\alpha.
\]
\item[(iii)] If the roots are distinct then Gevrey and ultradistributional well-posedness hold provided that 
\[
1\le s<1+\frac{\alpha}{1-\alpha}
\]
and
\[
1\le s\le1+\frac{\alpha}{1-\alpha},
\]
respectively.
\end{itemize}
\end{theorem}

It is our purpose in this paper to prove well-posedness of the Cauchy problem \eqref{intro_CP} 
when the coefficients are less than H\"older.

The first main idea now is to start from distributional coefficients $a_i$ and $b_i$, $i=1,\dots, n$, to regularise them by convolution with a suitable mollifier $\psi$ obtaining families of smooth functions $(a_{i,\eps})_\eps$ and $(b_{i,\eps})_\eps$, namely 
\begin{equation}\label{EQ:regs}
a_{i,\eps}=a\ast\psi_{\omega(\eps)} \textrm{ and } b_{i,\eps}=b_i\ast\psi_{\omega(\eps)}, 
\end{equation}
where  $\psi_{\omega(\eps)}(t)=\omega(\eps)^{-1}\psi(t/\omega(\eps))$ and $\omega(\eps)$ is a positive function converging to $0$ as $\eps\to 0$. It turns out that the nets $(a_{i,\eps})_\eps$ and $(b_{i,\eps})_\eps$ are $C^\infty$-\emph{moderate}, in the sense that their $C^\infty$-seminorms can be estimated by a negative power of $\eps$ (see \eqref{mod_C_inf}). More precisely, we will make use of the following notions of moderateness.

In the sequel, the notation $K\Subset\R^n$ means that $K$ is a compact set in $\R^n$.
\begin{definition}
\label{def_mod_intro}
\leavevmode
\begin{itemize}
\item[(i)] A net of functions $(f_\eps)_\eps\in C^\infty(\R^n)^{(0,1])}$ is $C^\infty$-moderate if for all $K\Subset\R^n$ and for all $\alpha\in\N^n$ there exist $N\in\N$ and $c>0$ such that
\[
\sup_{x\in K}|\partial^\alpha f_\eps(x)|\le c\eps^{-N},
\]
for all $\eps\in(0,1]$. 
\item[(ii)] A net of functions $(f_\eps)_\eps\in \gamma^s(\R^n)^{(0,1]}$ is $\gamma^s$-moderate if for all $K\Subset\R^n$ there exists a constant $c_K>0$ and there exists $N\in\N$ such that 
\[
|\partial^\alpha f_\eps(x)|\le c_K^{|\alpha|+1}(\alpha !)^s \eps^{-N-|\alpha|},
\]
for all $\alpha\in\N^n$, $x\in K$ and $\eps\in(0,1]$.
\item[(iii)] A net of functions $(f_\eps)_\eps\in C^\infty([0,T];\gamma^s(\R^n))^{(0,1]}$ is $C^\infty([0,T];\gamma^s(\R^n))$-moderate if for all $K\Subset\R^n$ there exist $N\in\N$, $c>0$ and, for all $k\in\N$ there exist $N_k>0$ and $c_k>0$ such that
\[
|\partial_t^k\partial^\alpha_x u_\eps(t,x)|\le c_k\eps^{-N_k} c^{|\alpha|+1}(\alpha !)^s \eps^{-N-|\alpha|},
\]
for all $\alpha\in\N^n$, for all $t\in[0,T]$, $x\in K$ and $\eps\in(0,1]$.
\end{itemize}
\end{definition}
We note that the conditions of moderateness are natural in the sense that regularisations of distributions or
ultradistributions are moderate, namely we can think that
\begin{equation}\label{EQ:incls}
\textrm{ compactly supported distributions } \E'(\R^n)\subset \{C^\infty \textrm{-moderate families}\}
\end{equation}
by the structure theorems for distributions, while also
the regularisations of the compactly supported Gevrey ultradistributions can be shown to be Gevrey-moderate.

Thus, while a solution to a Cauchy problems may not exist in the space on the left hand side of 
an inclusion like the one in \eqref{EQ:incls}, it may still exist (in a certain appropriate sense)
in the space on its right hand side. The moderateness assumption will be enough for our
purposes. However, we note that regularisation with standard Friedrichs mollifiers will not be
sufficient, hence the introduction of a family $\omega(\eps)$ in the above regularisations.

We can now introduce a notion of a `very weak solution' for the Cauchy problem \eqref{intro_CP}.
\begin{definition}
\label{def_vws}
Let $s\ge1$. The net $(u_\eps)_\eps\in C^\infty([0,T];\gamma^s(\R^n))$ is a very weak solution of order $s$ of the 
Cauchy problem \eqref{intro_CP} if there exist 
\begin{itemize}
\item[(i)] $C^\infty$-moderate regularisations $a_{i,\eps}$ and $b_{i,\eps}$ of the coefficients $a_i$ and $b_i$, respectively, for $i=1,\dots,n$, 
\item[(ii)] $\gamma^s$-moderate regularisations $g_{0,\eps}$ and $g_{1,\eps}$ of the initial data $g_0$ and $g_1$, 
respectively,
\end{itemize}
such that $(u_\eps)_\eps$ solves the regularised problem
\[
\begin{split}
D_t^2u(t,x)-\sum_{i=1}^n b_{i,\eps}(t)D_tD_{x_i}u(t,x)-\sum_{i=1}^n a_{i,\eps}(t)D_{x_i}^2u(t,x)&=0,\\
u(0,x)&=g_{0,\eps},\\
D_t u(0,x)&=g_{1,\eps},\\
\end{split}
\]
for all $\eps\in(0,1]$, and is $C^\infty([0,T];\gamma^s(\R^n))$-moderate.
\end{definition}
The main results of this paper can be summarised as the following solvability statement complemented
by the uniqueness and consistency in Theorem \ref{theo_vws2}.
\begin{theorem}
\label{theo_vws}
Let the coefficients $a_i, b_i$ of the Cauchy problem \eqref{intro_CP} be distributions 
with compact support included in $[0,T]$, such that $a_i, b_i$ are real-valued and $a_i\ge 0$ for all $i=1,\dots,n$. 
Let the Cauchy data $g_0, g_1$ be compactly supported distributions.
Then, the Cauchy problem \eqref{intro_CP} has a very weak solution of order $s$, for all $s>1$.
\end{theorem}
In fact, Theorem \ref{theo_vws} will be refined according to the regularity of the initial data. 
More precisely, we will distinguish between the following cases:
\begin{itemize}
\item[{\bf Case 1}:] distributional coefficients and Gevrey initial data;
\item[{\bf Case 2}:] distributional coefficients and smooth initial data;
\item[{\bf Case 3}:] distributional coefficients and distributional initial data.
\end{itemize}

The uniqueness and consistency result for very weak solutions 
of the Cauchy problem \eqref{intro_CP} is as follows.
We distinguish between Gevrey Cauchy data and the general distributional
Cauchy data:

\begin{theorem}
\label{theo_vws2}
Assume that the real-valued coefficients $a_i$ and $b_i$ are compactly supported, belong to $C^k([0,T])$ with $k\ge 2$ and 
that $a_i\ge 0$ for all 
$i=1,\dots,n$. Let $1<s<1+\frac{k}{2}$.

\begin{itemize}
\item
Let $g_0, g_1\in \gamma^s_c(\R^n)$. 
Then any very weak solution $(u_\eps)_\eps$ converges in the space
$C([0,T];\gamma^s(\R^n))$ as $\eps\to0$ to the unique classical solution in 
$C^2([0,T], \gamma^s(\R^n))$.
In particular, 
this limit exists and does not depend on the $C^\infty$-moderate regularisation of 
the coefficients.
\item
Let $g_0, g_1\in \E'(\R^n)$.
Then any very weak solution $(u_\eps)_\eps$ converges in the space
$C([0,T];\D'_{(s)}(\R^n))$ to the unique solution 
in $C^2([0,T],\D'_{(s)}(\R^n))$.
In particular, 
this limit exists and does not depend on the $C^\infty$-moderate regularisation of 
coefficients $a_i$ and $b_i$ and the Gevrey-moderate regularisation
of the initial data $g_0,g_1$. 
\end{itemize}
\end{theorem}

In Theorem \ref{theo_vws2}, we assume that $1<s<1+\frac{k}{2}$ in order to
make sure that the unique classical or ultradistributional solutions exist,
provided by Theorem \ref{THM:GR12}. 
Theorem \ref{theo_vws2} will follow from Theorem \ref{theo_consistency}.

The proof of Theorem \ref{theo_vws} relies on classical techniques for weakly hyperbolic equations (quasi-symmestriser, energy estimates, Gevrey-wellposedness, etc.) and ideas from generalised function theory (regularisation). In particular, proving the existence of a very weak solution coincides, by fixing the mollifiers,  with proving well-posedness of the corresponding Cauchy problem in a suitable space of Colombeau type. This space will be chosen according to the regularity of the initial data. So, the proof of Theorem \ref{theo_vws} will follow from the well-posedness results in Theorems  \ref{theo_CP_1}, \ref{theo_CP_2} and \ref{theo_CP_3}.

We note that the proof of Theorem \ref{theo_vws} actually provides us with 
a description of possible regularisations, in particular, of
functions $\omega(\eps)$ used in the regularisation of coefficients
in \eqref{EQ:regs}. Indeed, $\omega(\eps)$ will be of the type $c(\log(\eps^{-1}))^{-r}$ or of the type $c(\log(\eps^{-1}))^{-r_1}\eps^{-r_2}$,
for $c>0$ and $r,r_1,r_2>0$.

We note that the idea of considering regularisations of coefficients and solutions of partial differential equations 
in different senses has been seen in the literature. 
For example, after regularising (e.g. non-Lipschitz, H\"older) coefficients with a parameter $\epsilon$,
relating $\epsilon$ to some frequency zone in the energy estimate often yields the Gevrey or even
$C^\infty$ well-posedness (see e.g. Colombini, del Santo and Kinoshita
\cite{Colombini-del-Santo-Kinoshita:ASNS-2002} and other papers).
For less regularity, e.g. for hyperbolic equations with discontinuous
coefficients regularised families have been already considered by Hurd and Sattinger
\cite{Hurd-Sattinger}, with a subsequent analysis of limits of these regularisations in $L^2$ as
$\eps\to 0$.
An interesting result of well-posedness has been obtained for discontinuous and in general distributional coefficients in the Colombeau context by Lafon and Oberguggenberger \cite{LO:91}.
In their paper they proved that first order symmetric systems of differential equations with 
Colombeau coefficients and Colombeau initial data have a unique Colombeau solution under 
suitable logarithmic type assumptions on the principal part. This result, while it can be easily extended to 
pseudo-differential systems, cannot be directly applied to our equation, since
the system to which would can reduce our equations is in general, non-symmetric and non-strictly hyperbolic.  

It will be useful also to us to use the developed machinery of Colombeau algebras in the proofs.
Especially, this will provide an easy-to-get refinement of the uniqueness part of the corresponding statements.
However, we need to work in algebras of generalised functions based on regularisations with Gevrey
functions since smooth solutions do not have to exist due to multiplicities.

As mentioned above, we will employ quasi-symmetriser techniques, or more precisely, a parametrised version of the quasi-symmetriser seen in \cite{GR:12}. This is the topic of the next section.

\section{Parameter dependent quasi-symmetriser}
\label{sec_quasi}

In this paper, we will be applying the standard reduction of a scalar second order equation to
the $2\times 2$ system: setting
\[
u_j=D_t^{j-1}\lara{D_x}^{2-j}u,\qquad j=1,2,
\]
we transform the equation
\[
D_t^2u(t,x)-\sum_{i=1}^n b_i(t)D_tD_{x_i}u(t,x)-\sum_{i=1}^n a_i(t)D_{x_i}^2u(t,x)=0
\]
into the hyperbolic system
\beq
\label{system_hyp}
D_t\left(
                             \begin{array}{c}
                               u_1 \\
                                u_2 \\
                             \end{array}
                           \right)
= \left(
    \begin{array}{cc}
      0 & \lara{D_x}\\
      \sum_{i=1}^n a_i(t)D_{x_i}^2\lara{D_x}^{-1}& \sum_{i=1}^n b_i(t)D_{x_i} \\
           \end{array}
  \right)
  \left(\begin{array}{c}
                               u_1 \\
                               
                               u_2 \\
                             \end{array}
                           \right),
\eeq

We now assume that the equation coefficients are distributions with compact support contained in $[0,T]$. Since the formulation of \eqref{intro_eq} might be impossible due to issues related to the product of distributions, we replace \eqref{intro_eq} with a regularised equation. In other words, we regularise every $a_i$ and $b_i$ by convolution with a mollifier in $C^\infty_c(\R^n)$ and get nets of smooth functions as coefficients. More precisely, let $\psi\in C^\infty_c(\R)$, $\psi\ge 0$ with $\int\psi=1$ and let $\omega(\eps)$ be a positive function converging to $0$ as $\eps\to 0$. Define
\[
\psi_{\omega(\eps)}(t):=\frac{1}{\omega(\eps)}\psi\big(\frac{t}{\omega(\eps)}\big),
\]
\[
a_{i,\eps}(t):=(a_i\ast \psi_{\omega(\eps)})(t),\qquad t\in[0,T]
\]
and
\[
b_{i,\eps}(t):=(b_i\ast \psi_{\omega(\eps)})(t),\qquad t\in[0,T].
\]
By the structure theorem for compactly supported distributions, we have that there exists $L\in\N$ and $c>0$ such that
\[
|a_{i,\eps}(t)|\le c\,\omega(\eps)^{-L},\qquad |b_{i,\eps}(t)|\le c\,\omega(\eps)^{-L},
\]
for all $i=1,\dots,n$. Regularising the equation \eqref{intro_eq} means equivalently to regularise the system \eqref{system_hyp} as 
\[
D_t\left(
                             \begin{array}{c}
                               u_1 \\
                                u_2 \\
                             \end{array}
                           \right)
= \left(
    \begin{array}{cc}
      0 & \lara{D_x}\\
      \sum_{i=1}^n a_{i,\eps}(t)D_{x_i}^2\lara{D_x}^{-1}& \sum_{i=1}^n b_{i,\eps}(t)D_{x_i} \\
           \end{array}
  \right)
  \left(\begin{array}{c}
                               u_1 \\
                               
                               u_2 \\
                             \end{array}
                           \right),
                           \]
  with symbol matrix $A_{1,\eps}(t,\xi)$. Clearly, one can write $A_{1,\eps}(t,\xi)$ as $\lara{\xi}A_\eps(t,\xi)$, where
\[
A_\eps(t,\xi)=\left(
    \begin{array}{cc}
      0 & 1\\
      \sum_{i=1}^n a_{i,\eps}(t)\xi_i^2\lara{\xi}^{-2}& \sum_{i=1}^n b_{i,\eps}(t)\xi_i \lara{\xi}^{-1}\\
           \end{array}
  \right).
\]
The matrix $A_\eps(t,\xi)$ has eigenvalues
\beq
\label{eigen_eps}
\begin{split}
\lambda_{1,\eps}(t,\xi)&= \frac{1}{2}\biggl(\sum_{i=1}^n b_{i,\eps}(t)\xi_i\lara{\xi}^{-1}-\sqrt{\big(\sum_{i=1}^n b_{i,\eps}(t)\xi_i\lara{\xi}^{-1}\big)^2+4\sum_{i=1}^n a_{i,\eps}(t)\xi_i^2\lara{\xi}^{-2}}\biggl),\\
\lambda_{2,\eps}(t,\xi)&= \frac{1}{2}\biggl(\sum_{i=1}^n b_{i,\eps}(t)\xi_i\lara{\xi}^{-1}+\sqrt{\big(\sum_{i=1}^n b_{i,\eps}(t)\xi_i\lara{\xi}^{-1}\big)^2+4\sum_{i=1}^n a_{i,\eps}(t)\xi_i^2\lara{\xi}^{-2}}\biggl).
\end{split}
\eeq
Note that $\lambda_{1,\eps}\lara{\xi}$ and $\lambda_{2,\eps}\lara{\xi}$ are the roots of the characteristic polynomial 
\[
\tau^2-\sum_{i=1}^n b_{i,\eps}(t)\xi_i\tau-\sum_{i=1}^n a_{i,\eps}(t)\xi_i^2
\]
and fulfil the inequality
\[
\lambda_{1,\eps}(t,\xi)^2+\lambda_{2,\eps}(t,\xi)^2\le 2(\lambda_{1,\eps}(t,\xi)-\lambda_{2,\eps}(t,\xi))^2,
\]
employed by Kinoshita and Spagnolo in \cite{KS} to obtain Gevrey well-posedness for the corresponding Cauchy problem.

It is clear that the regularised equation \eqref{intro_eq} and the corresponding first order system have solutions $(u_\eps)_\eps$  and $(U_\eps)_\eps$, respectively, depending on the parameter $\eps\in(0,1]$. By Fourier transformation in $x$ the system 
\beq
\label{system_hyp_eps}
D_t U_\eps 
= \left(
    \begin{array}{cc}
      0 & \lara{D_x}\\
      \sum_{i=1}^n a_{i,\eps}(t)D_{x_i}^2\lara{D_x}^{-1}& \sum_{i=1}^n b_{i,\eps}(t)D_{x_i} \\
           \end{array}
  \right)
 U_\eps,
\eeq
where
\[
U_\eps=\left(
                             \begin{array}{c}
                               u_{1,\eps} \\
                                u_{2,\eps} \\
                             \end{array}
                           \right)=\left(
                             \begin{array}{c}
                               \lara{D_x}u_{\eps} \\
                                D_t u_{\eps} \\
                             \end{array}
                           \right)
\]
is transformed into
\beq
\label{eq_V_eps}
D_t V_\eps(t,\xi)=\lara{\xi}A_\eps(t,\xi)V_\eps(t,\xi),
\eeq
where $V_\eps(t,\xi)=(\mathcal{F}U_\eps(t,\cdot))(\xi)$. Finally, by regularising the initial data as well if needed (for instance in Case 3), we transform the Cauchy problem \eqref{intro_CP} into 
\[
\begin{split}
D_t V_\eps(t,\xi)&=\lara{\xi}A_\eps(t,\xi)V_\eps(t,\xi),\\
V_{0,\eps}&=(\mathcal{F}U_{\eps}(0,\cdot))(\xi)=\mathcal{F}(\lara{D_x}g_{0,\eps}, g_{1,\eps}).
\end{split}
\]
The well-posedness of this regularised Cauchy problem will be obtained by constructing a quasi-symmetriser for the matrix $A_\eps$ and the corresponding energy. Before proceeding with the technical details we recall some general basic facts. For more details see \cite{DS,KS}.

\subsection{The quasi-symmetriser: general theory}
Note that for $m\times m$ matrices $A_1$ and $A_2$ the notation $A_1\le A_2$ means $(A_1v,v)\le (A_2v,v)$ for all $v\in\C^m$ with $(\cdot,\cdot)$ the scalar product in $\C^m$.

Let $A(\lambda)$ be the $m\times m$ Sylvester matrix with real eigenvalues $\lambda_l$, i.e.,
\[
A(\lambda)=\left(
    \begin{array}{ccccc}
      0 & 1 & 0 & \dots & 0\\
      0 & 0 & 1 & \dots & 0 \\
      \dots & \dots & \dots & \dots & 1 \\
      -\sigma_m^{(m)}(\lambda) & -\sigma_{m-1}^{(m)}(\lambda) & \dots & \dots & -\sigma_1^{(m)}(\lambda) \\
    \end{array}
  \right),
\]
where
\[
\sigma_h^{(m)}(\lambda)=(-1)^h\sum_{1\le i_1<...<i_h\le m}\lambda_{i_1}...\lambda_{i_h}
\]
for all $1\le h\le m$. In the sequel we make use of the following notations: $\mathcal{P}_m$ for the class of permutations of $\{1,...,m\}$, $\lambda_\rho=(\lambda_{\rho_1},...,\lambda_{\rho_m})$ with $\lambda\in\R^m$ and $\rho\in\mathcal{P}_m$, $\pi_i\lambda=(\lambda_1,...,\lambda_{i-1},\lambda_{i+1},...,\lambda_m)$ and $\lambda'=\pi_m\lambda=(\lambda_1,...,\lambda_{m-1})$. Following Section 4 in \cite{KS} we have that the quasi-symmetriser is the Hermitian matrix
\[
Q^{(m)}_\delta(\lambda)=\sum_{\rho\in\mathcal{P}_m} P_\delta^{(m)}(\lambda_\rho)^\ast P_\delta^{(m)}(\lambda_\rho),
\]
where $\delta\in(0,1]$, $P_\delta^{(m)}(\lambda)=H^{(m)}_\delta P^{(m)}(\lambda)$, $H_\delta^{(m)}={\rm diag}\{\delta^{m-1},...,\delta,1\}$ and the matrix $P^{(m)}(\lambda)$ is defined inductively by $P^{(1)}(\lambda)=1$ and
\[
P^{(m)}(\lambda)=\left(
    \begin{array}{ccccc}
      \, & \, & \, & \, & 0\\
      \, & \, & P^{(m-1)}(\lambda') & \, & \vdots \\
      \, & \, & \, & \, & 0 \\
      \sigma_{m-1}^{(m-1)}(\lambda') & \dots & \dots & \sigma_1^{(m-1)}(\lambda') & 1 \\
    \end{array}
  \right).
\]
Note that $P^{(m)}(\lambda)$ is depending only on $\lambda'$. Finally, let $W^{(m)}_i(\lambda)$ denote the row vector
\[
\big(\sigma_{m-1}^{(m-1)}(\pi_i\lambda),...,\sigma_1^{(m-1)}(\pi_i\lambda),1\big),\quad 1\le i\le m,
\]
and let $\mathcal{W}^{(m)}(\lambda)$ be the matrix with row vectors $W^{(m)}_i$. The following proposition collects the main properties of the quasi-symmetriser $Q^{(m)}_\delta(\lambda)$. For a detailed proof we refer the reader to Propositions 1 and 2 in \cite{KS} and to Proposition 1 in \cite{DS}.
\begin{proposition}
\label{prop_qs}
\leavevmode
\begin{itemize}
\item[(i)] The quasi-symmetriser $Q_\delta^{(m)}(\lambda)$ can be written as
\[
Q_0^{(m)}(\lambda)+\delta^2 Q_1^{(m)}(\lambda)+...+\delta^{2(m-1)}Q_{m-1}^{(m)}(\lambda),
\]
where the matrices $Q^{(m)}_i(\lambda)$, $i=1,...,m-1,$ are nonnegative and Hermitian with
entries being symmetric polynomials in $\lambda_1,...,\lambda_m$.
\item[(ii)] There exists a function $C_m(\lambda)$ bounded for
bounded $|\lambda|$ such that
\[
C_m(\lambda)^{-1}\delta^{2(m-1)}I\le Q^{(m)}_\delta(\lambda)\le C_m(\lambda)I.
\]
\item[(iii)] We have
\[
|Q_\delta^{(m)}(\lambda)A(\lambda)-A(\lambda)^\ast Q_\delta^{(m)}(\lambda)|\le C_m(\lambda)\delta Q_\delta^{(m)}(\lambda).
\]
\item[(iv)] For any $(m-1)\times(m-1)$ matrix $T$ let $T^\sharp$ denote the $m\times m$ matrix
\[
\left(
    \begin{array}{cc}
    T & 0\\
    0 & 0 \\
    \end{array}
  \right).
\]
Then, $Q_\delta^{(m)}(\lambda)=Q_0^{(m)}(\lambda)+\delta^2\sum_{i=1}^m Q_\delta^{(m-1)}(\pi_i\lambda)^\sharp$.
\item[(v)] We have
\[
Q_0^{(m)}(\lambda)=(m-1)!\mathcal{W}^{(m)}(\lambda)^\ast \mathcal{W}^{(m)}(\lambda).
\]
\item[(vi)] We have
\[
\det Q_0^{(m)}(\lambda)=(m-1)!\prod_{1\le i<j\le m}(\lambda_i-\lambda_j)^2.
\]
\item[(vii)] There exists a constant $C_m$ such that
\[
q_{0,11}^{(m)}(\lambda)\cdots q_{0,mm}^{(m)}(\lambda)\le C_m\prod_{1\le i<j\le m}(\lambda^2_i+\lambda^2_j).
\]
\end{itemize}
\end{proposition}
We finally recall that a family $\{Q_\alpha\}$ of nonnegative Hermitian matrices is called \emph{nearly diagonal} if there exists a positive constant $c_0$ such that
\[
Q_\alpha\ge c_0\,{\rm diag}\,Q_\alpha
\]
for all $\alpha$, with ${\rm diag}\,Q_\alpha ={\rm diag}\{q_{\alpha,11},...,q_{\alpha, mm}\}$. The following linear algebra result is proven in \cite[Lemma 1]{KS}.
\begin{lemma}
\label{lem_old}
Let $\{Q_\alpha\}$ be a family of nonnegative Hermitian $m\times m$ matrices such that $\det Q_\alpha>0$ and
\[
\det Q_\alpha \ge c\, q_{\alpha,11}q_{\alpha,22}\cdots q_{\alpha,mm}
\]
for a certain constant $c>0$ independent of $\alpha$. Then,
\[
Q_\alpha\ge c\, m^{1-m}\,{\rm diag}\,Q_\alpha
\]
for all $\alpha$, i.e., the family $\{Q_\alpha\}$ is nearly diagonal.
\end{lemma}
Lemma \ref{lem_old} is employed to prove that the family  $Q_\delta^{(m)}(\lambda)$ of quasi-sym\-me\-tri\-sers defined above is nearly diagonal when $\lambda$ belongs to a suitable set. The following statement is proven in \cite[Proposition 3]{KS}.
\begin{proposition}
\label{prop_SM}
For any $M>0$ define the set
\[
\mathcal{S}_M=\{\lambda\in\R^m:\, \lambda_i^2+\lambda_j^2\le M (\lambda_i-\lambda_j)^2,\quad 1\le i<j\le m\}.
\]
Then the family of matrices $\{Q_\delta^{(m)}(\lambda):\, 0<\delta\le 1, \lambda\in\mathcal{S}_M\}$ is nearly diagonal.
\end{proposition}
We conclude this section with a result on nearly diagonal matrices depending on three parameters (i.e. $\delta, t, \xi$)
which will be crucial in the next section. Note that this is a straightforward
extension of Lemma 2 in \cite{KS} valid for two parameter
(i.e. $\delta, t$) dependent matrices.
\begin{lemma}
\label{lem_new}
Let $\{ Q^{(m)}_\delta(t,\xi): 0<\delta\le 1, 0\le t\le T, \xi\in\R^n\}$ be a nearly diagonal family of coercive Hermitian matrices of class ${C}^k$ in $t$, $k\ge 1$. Then, there exists a constant $C_T>0$ such that for any non-zero
continuous function $V:[0,T]\times\R^n\to \C^m$
we have
\[
\int_{0}^T \frac{|(\partial_t Q^{(m)}_\delta(t,\xi)
V(t,\xi),V(t,\xi))|}{(Q^{(m)}_\delta(t,\xi)V(t,\xi),V(t,\xi))^{1-1/k}
|V(t,\xi)|^{2/k}}\, dt\le C_T
\Vert Q^{(m)}_\delta(\cdot,\xi)\Vert^{1/k}_{{C}^k([0,T])}
\]
for all $\xi\in\Rn.$
\end{lemma}
\subsection{The quasi-symmetriser of the matrix $A_\eps$}
We now focus on the matrix $A_\eps$ corresponding to the Cauchy problem we are studying. It is clear that we will get a family of quasi-symmetrisers $(Q^{(2)}_{\delta}(\lambda_\eps))_\eps$, where $\lambda_\eps=(\lambda_{1,\eps},\lambda_{2,\eps})$. More precisely, by direct computations we get
\[
Q^{(2)}_{\delta}(\lambda_\eps)=\left(
    \begin{array}{cc}
      \lambda_{1,\eps}^2+\lambda_{2,\eps}^2 & -(\lambda_{1,\eps}+\lambda_{2,\eps})\\
      -(\lambda_{1,\eps}+\lambda_{2,\eps}) & 2 \\
           \end{array}
  \right) + 2\delta^2 \left(
    \begin{array}{cc}
      1 & 0\\
      0 & 0 \\
           \end{array}
  \right),
\]
where $\lambda_{1,\eps}$ and $\lambda_{2,\eps}$ are defined as in \eqref{eigen_eps}. Thus,
\begin{multline*}
Q^{(2)}_{\delta}(\lambda_\eps)=\left(
    \begin{array}{cc}
     \big(\sum_{i=1}^n b_{i,\eps}(t)\xi_i\big)^2\lara{\xi}^{-2}+2\sum_{i=1}^n a_{1,\eps}(t)\xi_i^2\lara{\xi}^{-2} & -\sum_{i=1}^n b_{i,\eps}(t)\xi_i\lara{\xi}^{-1}\\[0.2cm]
      -\sum_{i=1}^n b_{i,\eps}(t)\xi_i\lara{\xi}^{-1} & 2 \\
           \end{array}
  \right)\\
   + 2\delta^2 \left(
    \begin{array}{cc}
      1 & 0\\
      0 & 0 \\
           \end{array}
  \right).
\end{multline*}
Note that from the formula \eqref{eigen_eps}, $\lambda_{1,\eps}$ and $\lambda_{2,\eps}$ are nets of smooth functions fulfilling the estimate
\beq
\label{est_lambda}
|\partial_t^{(k)}\lambda_{i,\eps}(t,\xi)|\le c_k\omega(\eps)^{-L-k},
\eeq
for all $k\in\N$, for $t\in[0,1]$, $\xi\in\R^n$ and $\eps\in(0,1]$. 
Finally, since $\lambda_{1,\eps}$ and $\lambda_{2,\eps}$ in \eqref{eigen_eps} are roots of a second order equation, they fulfil
\beq
\label{KS_cond}
\lambda_{1,\eps}^2(t,\xi)+\lambda_{2,\eps}^2(t,\xi)\le 2(\lambda_{1,\eps}(t,\xi)-\lambda_{2,\eps}(t,\xi))^2,
\eeq
so the condition on the roots used in \cite{KS} and in \cite{GR:12} is trivially fulfilled with constant $M=2$ (see (6) in \cite{GR:12}).

By analysing Lemma \ref{lem_old} and Proposition \ref{prop_qs} in this particular case we get the following results on the quasi-symmetriser $Q^{(2)}_\delta(\lambda_\eps)$.
\begin{proposition}
\label{prop_qs_2}
Let $Q^{(2)}_\delta(\lambda_\eps)$ as defined above. Then,
\beq
\label{est_diag_below}
(Q^{(2)}_\delta(\lambda_\eps)V,V)\ge \frac{1}{8} (Q^{(2)}_{\delta,\Delta}(\lambda_\eps)V,V),
\eeq
where $Q^{(2)}_{\delta,\Delta}(\lambda_\eps)$ is the diagonal part of the matrix $Q^{(2)}_\delta(\lambda_\eps)$. In addition, there exists a constant $C_2>0$ such that
\begin{itemize}
\item[(i)]  $C_2^{-1}\omega(\eps)^{2L}\delta^2 I\le Q^{(2)}_\delta(\lambda_\eps(t,\xi))\le C_2\omega(\eps)^{-2L} I$,\\
\item[(ii)] $|((Q^{(2)}_\delta(\lambda_\eps)A_\eps(t,\xi)-A_\eps(t,\xi)^\ast Q^{(2)}_\delta(\lambda_\eps))V,V)|\le C_2\delta(Q^{(2)}_\delta(\lambda_\eps)V,V)$,\\
\end{itemize}
for all $\delta>0$, $\eps\in(0,1]$, $t\in[0,T]$, $\xi\in\R^n$ and $V\in\C^2$.  
\end{proposition}
\begin{proof}
By direct computations and by \eqref{KS_cond} we have that
\begin{multline*}
\det Q^{(2)}_\delta(\lambda_\eps)=(\lambda_1-\lambda_2)^2+4\delta^2\ge \frac{1}{2}(\lambda_1(t,\xi)^2+\lambda_2(t,\xi)^2)+4\delta^2\\
\ge\frac{2}{4}(\lambda_1(t,\xi)^2+\lambda_2(t,\xi)^2+2\delta^2)
=\frac{1}{4}q_{\delta,11}(\lambda_\eps)q_{\delta,22}(\lambda_\eps).
\end{multline*}
Note that the estimate below is uniform in $\eps$ and $\delta$. Hence, Lemma \ref{lem_old} yields
\[
(Q^{(2)}_\delta(\lambda_\eps)V,V)\ge \frac{1}{8} (Q^{(2)}_{\delta,\Delta}(\lambda_\eps)V,V).
\]
We pass now to prove assertion (i). We have that
\begin{multline*}
(Q^{(2)}_{\delta}(\lambda_\eps)V,V)=(\lambda_{1,\eps}^2+\lambda_{2,\eps}^2)|V_1|^2-2(\lambda_{1,\eps}+\lambda_{2,\eps})\Re (\overline{V_1}V_2)+2\delta^2 |V_1|^2+ 2|V_2|^2\\
=|\lambda_{1,\eps}V_1-V_2|^2+|\lambda_{2,\eps}V_1-V_2|^2+2\delta^2|V_1|^2.
\end{multline*}
It follows that if $|V_1|^2\ge \gamma\omega(\eps)^{2L}|V_2|^2$, with $0\le\gamma\le1$ we have that
\beq
\label{est_below_1}
(Q^{(2)}_{\delta}(\lambda_\eps)V,V)\ge 2\delta^2|V_1|^2=\delta^2(|V_1|^2+|V_1|^2)\ge C^{-1}_2\omega(\eps)^{2L}\delta^2(|V_1|^2+|V_2|^2).
\eeq
On the other hand, recalling from \eqref{est_lambda} that $|\lambda_{i,\eps}(t,\xi)|^2\omega(\eps)^{2L}\le c$ uniformly in variables and parameter for $i=1,2$,  if $|V_1|^2\le \gamma\omega(\eps)^{2L}|V_2|^2$,  we can write
\begin{multline*}
(Q^{(2)}_{\delta}(\lambda_\eps)V,V)=|\lambda_{1,\eps}V_1-V_2|^2+|\lambda_{2,\eps}V_1-V_2|^2+2\delta^2|V_1|^2\\
\ge \frac{1}{2}|V_2|^2-\lambda_{1,\eps}^2|V_1|^2+\frac{1}{2}|V_2|^2-\lambda_{2,\eps}^2|V_1|^2+2\delta^2|V_1|^2\\
\ge |V_2|^2-(\lambda_{1,\eps}^2+\lambda_{2,\eps}^2)\gamma\omega(\eps)^{2L}|V_2|^2+2\delta^2|V_1|^2\\
\ge (|V_2|^2-c\gamma|V_2|^2)+2\delta^2|V_1|^2.
\end{multline*}
So, choosing $\gamma$ sufficiently small and for $C_2$ big enough we have that 
\beq
\label{est_below_2}
(Q^{(2)}_{\delta}(\lambda_\eps)V,V)\ge C^{-1}_2\delta^2(|V_1|^2+|V_2|^2).
\eeq
Combining \eqref{est_below_1} with \eqref{est_below_2} we conclude that
\beq
\label{quasi-sym-below}
(Q^{(2)}_{\delta}(\lambda_\eps)V,V)\ge C^{-1}_2\omega(\eps)^{2L}\delta^2(|V_1|^2+|V_2|^2),
\eeq
for all $V\in\C^2$. Finally, from \eqref{est_lambda} we have that 
\[
(Q^{(2)}_{\delta}(\lambda_\eps)V,V)\le C_2\omega(\eps)^{-2L}(|V_1|^2+|V_2|^2),
\]
proving in this way that assertion (i) holds.

We now want to prove assertion (ii). We begin by computing $Q^{(2)}_\delta(\lambda_\eps)A_\eps(t,\xi)-A_\eps(t,\xi)^\ast Q^{(2)}_\delta(\lambda_\eps)$. We get
\[
Q^{(2)}_\delta(\lambda_\eps)A_\eps(t,\xi)-A_\eps(t,\xi)^\ast Q^{(2)}_\delta(\lambda_\eps)=\left(
    \begin{array}{cc}
    0 & 2\delta^2\\
    -2\delta^2 & 0 \\
    \end{array}
  \right)
\]
and therefore
\[
((Q^{(2)}_\delta(\lambda_\eps)A_\eps(t,\xi)-A_\eps(t,\xi)^\ast Q^{(2)}_\delta(\lambda_\eps))V,V)=2\delta^2(V_2\overline{V_1}-V_1\overline{V_2})=4i\delta^2\Im \overline{V_1}V_2.
\]
This means that $((Q^{(2)}_\delta(\lambda_\eps)A_\eps(t,\xi)-A_\eps(t,\xi)^\ast Q^{(2)}_\delta(\lambda_\eps))V,V)$ does not depend on the eigenvalues $\lambda_\eps=(\lambda_{1,\eps},\lambda_{2,\eps})$, or in other words, by replacing $Q^{(2)}_\delta(\lambda_\eps)$ with $Q^{(2)}_{\delta,\Delta}(\lambda_\eps)$ we can preliminary prove
\begin{multline*}
|((Q^{(2)}_\delta(\lambda_\eps)A_\eps(t,\xi)-A_\eps(t,\xi)^\ast Q^{(2)}_\delta(\lambda_\eps))V,V)|\\
=|((Q^{(2)}_{\delta,\Delta}(\lambda_\eps)A_\eps(t,\xi)-A_\eps(t,\xi)^\ast Q^{(2)}_{\delta,\Delta}(\lambda_\eps))V,V)|
\le 2\delta (Q^{(2)}_{\delta,\Delta}(\lambda_\eps)V,V).
\end{multline*}
This is easily done. Indeed,
\[
|((Q^{(2)}_\delta(\lambda_\eps)A_\eps(t,\xi)-A_\eps(t,\xi)^\ast Q^{(2)}_\delta(\lambda_\eps))V,V)|\le 2\delta 2\delta|V_1||V_2|
\]
and,  
\[
(Q^{(2)}_{\delta,\Delta}(\lambda_\eps)V,V)\ge 2\delta^2|V_1|^2+ 2|V_2|^2.
\] 
It follows that
\[
2\delta 2\delta|V_1||V_2|\le 2\delta(\delta^2|V_1|^2+|V_2|^2),
\]
thus
\beq
\label{est_ii_1}
|((Q^{(2)}_{\delta}(\lambda_\eps)A_\eps(t,\xi)-A_\eps(t,\xi)^\ast Q^{(2)}_\delta(\lambda_\eps))V,V)|\\
\le \delta (Q^{(2)}_{\delta,\Delta}(\lambda_\eps)V,V).
\eeq
The proof of assertion (ii)  is completed by combining \eqref{est_ii_1} with \eqref{est_diag_below}.

\end{proof}
Note that
\[
\frac{1}{8}(Q^{(2)}_{\delta,\Delta}(\lambda_\eps)V,V)\le (Q^{(2)}_{\delta}(\lambda_\eps)V,V)\le 2 (Q^{(2)}_{\delta,\Delta}(\lambda_\eps)V,V).
\]
Indeed,
\begin{multline*}
(Q^{(2)}_{\delta}(\lambda_\eps)V,V)=(\lambda_{1,\eps}^2+\lambda_{2,\eps}^2)|V_1|^2-2(\lambda_{1,\eps}+\lambda_{2,\eps})\Re (\overline{V_1}V_2)+2\delta^2 |V_1|^2+ 2|V_2|^2\\
\le 2(\lambda_{1,\eps}^2|V_1|^2+|V_2|^2)+2(\lambda_{2,\eps}^2|V_1|^2+|V_2|^2)+2\delta^2|V_1|^2\\
\le 2(\lambda_{1,\eps}^2+\lambda_{2,\eps}^2+2\delta^2)|V_1|^2+4|V_2|^2
=2 (Q^{(2)}_{\delta,\Delta}(\lambda_\eps)V,V).
\end{multline*}

Adopting the notations of \cite{KS} we then have that the bound from below (19) in \cite{KS} is fulfilled with $c_0=\frac{1}{8}$. This means that the family of matrices
\begin{multline*}
\{ Q^{(2)}_{\delta,\eps}(t,\xi):=
Q^{(2)}_\delta(\lambda_\eps),\\ \lambda_\eps(t,\xi)=(\lambda_{1,\eps}(t,\xi),\lambda_{2,\eps}(t,\xi)),\, t\in[0,T],\, \xi\in\R^n,\, \delta\in(0,1],\, \eps\in(0,1]\}
\end{multline*}
is nearly diagonal.

A careful analysis of the proof of Lemma 2 in \cite{KS} allows us to extend Lemma \ref{lem_new} to the family of quasi-symmetrisers $(Q_\delta^{(2)}(\lambda_\eps))_\eps$.  The constant $C_T=c_0^{-(1-1/k)}$  in Lemma 2 is in our case equal to $(1/8)^{-(1-1/k)}$.
\begin{lemma}
\label{lem_new_2}
Let $\{ Q^{(2)}_{\delta,\eps}(t,\xi): 0<\delta\le1, 0<\eps\le 1, 0\le t\le T, \xi\in\R^n\}$ be  the nearly diagonal family of quasi-symmetrisers introduced above. Then, for any
continuous function $V:[0,T]\times\R^n\to \C^2$, $V\neq 0$,
we have
\[
\int_{0}^T \frac{|(\partial_t Q^{(2)}_{\delta,\eps}(t,\xi)
V(t,\xi),V(t,\xi))|}{(Q^{(2)}_{\delta,\eps}(t,\xi)V(t,\xi),V(t,\xi))^{1-1/k}
|V(t,\xi)|^{2/k}}\, dt\le C_T
\Vert Q^{(2)}_{\delta,\eps}(\cdot,\xi)\Vert^{1/k}_{{C}^k([0,T])}
\]
for all $\xi\in\R^n$, $\delta\in(0,1]$ and $\eps\in(0,1]$.
\end{lemma}
We are now ready to prove the well-posedness of the Cauchy problem \eqref{intro_CP}. This will consist of two parts: 
\begin{enumerate}
\item choice of the framework,
\item  energy estimates.
\end{enumerate}
We begin by considering {\bf Case 1}: {\bf distributional coefficients and Gevrey initial data}

\section{Case 1: well-posedness for Gevrey initial data }
\label{sec_case1}
We want to prove the well-posedness of the Cauchy problem \eqref{intro_CP} when the coefficients of the equation are distributions with compact support and the initial data are compactly supported Gevrey functions. This will be achieved in a suitable algebra of Colombeau type containing the usual Gevrey classes as subalgebras. We start by developing these objects.

\subsection{Gevrey-moderate families}

We begin by investigating the convolution of a compactly supported Gevrey function with a mollifier $\varphi\in\S(\R^n)$ with $\int\varphi(x)\, dx=1$ and $\int x^\alpha\varphi(x)\, dx=0$ for all $\alpha\neq 0$ and $\varphi_\eps(x):=\eps^{-n}\varphi(x/\eps)$. The following holds:
\begin{proposition}
\label{prop_reg_gevrey_sim}
Let $\sigma> 1$. Let $u\in \gamma_c^\sigma(\R^n)$ and let $\varphi$ be a mollifier as above. Then
\begin{itemize}
\item[(i)] there exists $c>0$ such that 
\[
|\partial^\alpha(u\ast\varphi_\eps)(x)|\le c^{|\alpha|+1}(\alpha!)^\sigma
\]
for all $\alpha\in\N^n$, $x\in\R^n$ and $\eps\in(0,1]$;
\item[(ii)] there exists $c>0$ and for all $q\in\N$ a constant $c_q>0$ such that
\[
|\partial^\alpha(u\ast\varphi_\eps-u)(x)|\le c_q c^{|\alpha|+1}(\alpha!)^\sigma \eps^q,
\]
for all $\alpha\in\N^n$, $x\in\R^n$ and $\eps\in(0,1]$;
\item[(iii)] there exist $c, c'>0$ such that 
\[
|\widehat {u\ast\varphi_\eps}(\xi)|\le c'\,\esp^{-c\lara{\xi}^{\frac{1}{\sigma}}},
\]
for all $\xi\in\R^n$ and $\eps\in(0,1]$.
\end{itemize}
\end{proposition}
\begin{proof}
\leavevmode
\begin{itemize}
\item[(i)] By convolution with the mollifier $\varphi_\eps$ and straightforward estimates we obtain
\[
|\partial^\alpha (u\ast\varphi_\eps)(x)|=|\partial^\alpha u\ast\varphi_\eps(x)|\le \int_{\R^n}|\partial^\alpha u(x-\eps z)||\varphi(z)|\, dz\le c^{|\alpha|+1}(\alpha!)^\sigma,
\]
for all $\alpha\in\N^n$, $x\in\R^n$ and $\eps\in(0,1]$.
\item[(ii)] Analogously, by Taylor expansion and the properties of the mollifier $\varphi$ (in particular since $\int x^\alpha\varphi(x)\, dx=0$ for all $\alpha\neq 0$) we get for any $q\in\N$ the following estimate:
\begin{multline*}
|\partial^\alpha(u\ast\varphi_\eps-u)(x)|=|(\partial^\alpha u\ast\varphi_\eps-\partial^\alpha u)(x)|=\biggr|\int_{\R^n} (\partial^\alpha u(x-\eps z)-\partial^\alpha u(x))\varphi(z)\, dz\biggl|\\
=\biggr|\int_{\R^n}\sum_{|\beta|=q+1}\frac{\partial^{\alpha+\beta} u(x-\eps\theta z)}{\beta!}(\eps z)^\beta \varphi(z)\, dz\biggl|\le \int_{\R^n}\sum_{|\beta|=q+1}\frac{|\partial^{\alpha+\beta} u(x-\eps\theta z)|}{\beta!}\eps^{q+1}|z^\beta \varphi(z)|\, dz\\
\le \eps^{q+1} c^{|\alpha|+q+2}\sum_{|\beta|=q+1}\frac{((\alpha+\beta)!)^\sigma}{\beta!}\int_{\R^n}\sum_{|\beta|=q+1}|z^\beta \varphi(z)|\, dz\\
\le \eps^{q+1} c^{|\alpha|+q+2}\sum_{|\beta|=q+1}\frac{2^{\sigma|\alpha|+\sigma|\beta|}(\alpha!)^\sigma(\beta!)^\sigma}{\beta!}\int_{\R^n}\sum_{|\beta|=q+1}|z^\beta \varphi(z)|\, dz
\le c_q {\wt{c}}^{\,|\alpha|+1}(\alpha!)^\sigma \eps^q.
\end{multline*}
Note that the estimate above holds for all $\alpha\in\N^n$ and $q\in\N$ uniformly in $x\in\R^n$ and $\eps\in(0,1]$.
\item[(iii)] By Fourier transform we get that
\[
\widehat {u\ast\varphi_\eps}(\xi)=\widehat{u}(\xi)\widehat{\varphi_\eps}(\xi)=\widehat{u}(\xi)\widehat{\varphi}(\eps\xi)
\]
and therefore since $u\in \gamma_c^\sigma(\R^n)$ and $\varphi\in\S(\R^n)$ the third assertion is trivial.
\end{itemize}
\end{proof}
In Definition \ref{def_mod_intro} we introduced the notion of a moderate net, i.e.,  a net of functions $(f_\eps)_\eps\in \gamma^\sigma(\R^n)^{(0,1]}$ is $\gamma^s$-moderate if for all $K\Subset\R^n$ there exists a constant $c_K>0$ and there exists $N\in\N$ such that 
\[
|\partial^\alpha f_\eps(x)|\le c_K^{|\alpha|+1}(\alpha !)^\sigma \eps^{-N-|\alpha|},
\]
for all $\alpha\in\N^n$, $x\in K$ and $\eps\in(0,1]$. 

Analogously one can talk of $\gamma^\sigma$-negligible nets.
\begin{definition}
\label{def_negl_net}

Let $\sigma\ge 1$. We say that $(u_\eps)_\eps$ is \emph{$\gamma^\sigma$-negligible} if for all $K\Subset\R^n$ and for all $q\in\N$ there exists a constant $c_{q,K}>0$ such that 
\[
|\partial^\alpha u_\eps(x)|\le c_{q,K}^{|\alpha|+1}(\alpha !)^\sigma \eps^{q-|\alpha|},
\]
for all $\alpha\in\N^n$, $x\in K$ and $\eps\in(0,1]$.
\end{definition}
We can now prove the following proposition.
\begin{proposition}
\label{prop_reg_gevrey_mod}
\leavevmode
\begin{itemize}
\item[(i)] If $(u_\eps)_\eps$ is $\gamma^\sigma$-moderate and there exists $K\Subset\R^n$ such that $\supp\, u_\eps\subseteq K$ for all $\eps\in(0,1]$ then there exist $c,c'>0$ and $N\in\N$ such that
\beq
\label{star1}
|\widehat{u_\eps}(\xi)|\le c'\eps^{-N}\esp^{-c\eps^{\frac{1}{\sigma}}\lara{\xi}^{\frac{1}{\sigma}}},
\eeq
for all $\xi\in\R^n$.
\item[(ii)] If $(u_\eps)_\eps$ is $\gamma^\sigma$-negligible and there exists $K\Subset\R^n$ such that $\supp\, u_\eps\subseteq K$ for all $\eps\in(0,1]$ then there exists $c>0$ and for all $q>0$ there exists $c_q>0$  such that
\beq
\label{star2}
|\widehat{u_\eps}(\xi)|\le c_q\eps^{q}\esp^{-c\eps^{\frac{1}{\sigma}}\lara{\xi}^{\frac{1}{\sigma}}},
\eeq
for all $\xi\in\R^n$.

\item[(iii)] If $(u_\eps)_\eps$ is a net of tempered distributions with $(\widehat{u_\eps})_\eps$ satisfying \eqref{star1} then $(u_\eps)_\eps$ is $\gamma^s$-moderate.
\item[(iv)] If $(u_\eps)_\eps$ is a net of tempered distributions with $(\widehat{u_\eps})_\eps$ satisfying \eqref{star2} then $(u_\eps)_\eps$ is $\gamma^s$-negligible.
\end{itemize}
\end{proposition}
\begin{proof}
(i) By elementary properties of the Fourier transform and since $\supp\, u_\eps\subseteq K$ for all $\eps$ we have that
\beq
\label{est_a}
|\xi^\alpha \widehat{u_\eps}(\xi)|=|\mathcal{F}({D^\alpha(u_\eps)})(\xi)|\le \int_K|\partial^\alpha u_\eps(x)|\, dx\le C^{|\alpha|+1}(\alpha !)^\sigma \eps^{-|\alpha|-N},
\eeq
for all $\alpha\in\N^n$ and $\xi\in\R^n$. Let us now write $\lara{\xi}^{2M}|\widehat{u_\eps}(\xi)|^2$ as
\[
\sum_{k\le M}\binom{M}{k}|\xi|^{2k}|\widehat{u_\eps}(\xi)|^2=\sum_{k\le M}\binom{M}{k}\sum_{|\alpha|\le k}c_\alpha\xi^{2\alpha}|\widehat{u_\eps}(\xi)|^2,
\]
where $c_\alpha>0$. Hence, from \eqref{est_a} we have
\[
|\xi^{\alpha} \widehat{u_\eps}(\xi)|^2\le C^{2|\alpha|+2}(\alpha !)^{2\sigma}\eps^{-2|\alpha|-2N}
\]
and, therefore, from $\alpha!\le |\alpha|^{|\alpha|}$, we conclude
\[
\lara{\xi}^{2M}|\widehat{u_\eps}(\xi)|^2\le c_M C^{2M+2}M^{2\sigma M}\eps^{-2M-2N}.
\]
It is clear that this last estimate implies
\[
\lara{\xi}^{M}|\widehat{u_\eps}(\xi)|\le c'_M C^{M+1}M^{\sigma M}\eps^{-M-N},
\]
for all $M\in\N$, uniformly in $\eps\in(0,1]$ and $\xi\in\R^n$. Note that by direct computations on the binomial coefficients  one can see that the constant $c'_M$ is of the type $C'^{M}$ so
\[
\lara{\xi}^{M}|\widehat{u_\eps}(\xi)|\le C^{M+1}M^{\sigma M}\eps^{-M-N}\le C^{M+1}\esp^{\sigma M}(M!)^\sigma\eps^{-M-N},
\]
for some suitable constant $C>0$. It follows that
\[
\lara{\xi}^{\frac{M}{\sigma}}|\widehat{u_\eps}(\xi)|^{\frac{1}{\sigma}}\le C^{\frac{M}{\sigma}+\frac{1}{\sigma}}\esp^M M!\, \eps^{-\frac{M}{\sigma}-\frac{N}{\sigma}}\le 2^{-M} C^{\frac{M}{\sigma}+\frac{1}{\sigma}}2^M\esp^M M!\, \eps^{-\frac{M}{\sigma}-\frac{N}{\sigma}},
\]
and therefore introducing a suitable constant $\nu>0$ (depending on $\sigma$) we have that
\[
\sum_M |\widehat{u_\eps}(\xi)|^{\frac{1}{\sigma}}\biggl( \nu\lara{\xi}^{\frac{1}{\sigma}}\eps^{\frac{1}{\sigma}} \biggr)^M\frac{1}{M!}\le \sum_M 2^{-M} \eps^{-\frac{N}{\sigma}},
\]
for all $\xi\in\R^n$ and $\eps>0$. 
Concluding, recognising the Taylor series of an exponential in the previous formula, we arrive at
\[
|\widehat{u_\eps}(\xi)|\le c'\,\esp^{-c\lara{\xi}^{\frac{1}{\sigma}}\eps^{\frac{1}{\sigma}}}\eps^{-N},
\]
for a suitable constants $c,c'>0$ as desired.

(ii) The proof in (i) can be repeated for a $\gamma^\sigma_c$-negligible net $(u_\eps)_\eps$. From the assumption of negligibility it is immediate to see that the estimate
\[
|\widehat{u_\eps}(\xi)|\le c_q\eps^{q}\esp^{-c\eps^{\frac{1}{\sigma}}\lara{\xi}^{\frac{1}{\sigma}}},
\]
holds uniformly in $\xi$ and $\eps$.

(iii) If $(u_\eps)_\eps$ is a net of tempered distributions satisfying (i) then by the Fourier characterisation of Gevrey functions $(u_\eps)_\eps$ is a net of Gevrey functions of order $\sigma$. More precisely, 
\begin{multline}
\label{est_b}
|\partial^\alpha u_\eps(x)|=|\partial^\alpha\mathcal{F}^{-1}(\widehat{u_\eps})(x)|\le c\eps^{-N}\int_{\R^n}|\xi^\alpha|\esp^{-c\eps^{\frac{1}{\sigma}}\lara{\xi}^{\frac{1}{\sigma}}}\, d\xi\\
 =c\eps^{-N}\int_{\R^n}\esp^{-\frac{c}{2}\eps^{\frac{1}{\sigma}}\lara{\xi}^{\frac{1}{\sigma}}}d\xi\biggl(\sup_{\xi\in\R^n}|\xi^\alpha|\esp^{-\frac{c}{2}\eps^{\frac{1}{\sigma}}\lara{\xi}^{\frac{1}{\sigma}}}\biggr)\\
\le c\eps^{-N}\eps^{-n}\int_{\R^n}\esp^{-\frac{c}{2}|\xi|^{\frac{1}{\sigma}}}d\xi
\biggl(\sup_{\xi\in\R^n}|\xi^\alpha|\esp^{-\frac{c}{2}\eps^{\frac{1}{\sigma}}|\xi|^{\frac{1}{\sigma}}}\biggr)
\le c'\eps^{-N-n}\sup_{\xi\in\R^n}|\xi^\alpha|\esp^{-\frac{c}{2}\eps^{\frac{1}{\sigma}}\lara{\xi}^{\frac{1}{\sigma}}}.
\end{multline}
Clearly, 
\[
\sup_{|\xi|\le 1}|\xi^\alpha|\esp^{-\frac{c}{2}\eps^{\frac{1}{\sigma}}\lara{\xi}^{\frac{1}{\sigma}}}\le 1.
\]
Assume now that $|\xi|\ge 1$. Hence 
\[
\sup_{|\xi|\ge 1}|\xi^\alpha|\esp^{-\frac{c}{2}\eps^{\frac{1}{\sigma}}\lara{\xi}^{\frac{1}{\sigma}}}\le \sup_{|\xi|\ge 1}|\xi^\alpha|\esp^{-\frac{c}{2}\eps^{\frac{1}{\sigma}}{|\xi|}^{\frac{1}{\sigma}}}.
\]
Note that there exists a constant $c_\sigma>0$ such that
\begin{multline}
\label{est_c}
|\xi^\alpha|\esp^{-\frac{c}{2}\eps^{\frac{1}{\sigma}}|\xi|^{\frac{1}{\sigma}}}=\eps^{-|\alpha|}|\eps\xi|^{\frac{|\alpha|\sigma}{\sigma}}\esp^{-\frac{c}{2}\eps^{\frac{1}{\sigma}}|\xi|^{\frac{1}{\sigma}}}
=\eps^{-|\alpha|}\biggl(|\eps\xi|^{\frac{|\alpha|}{\sigma}}  \esp^{-\frac{c}{2\sigma}|\eps\xi|^{\frac{1}{\sigma}}}\biggr)^\sigma\\
=\eps^{-|\alpha|}\biggl(\biggl(|\eps\xi|^{\frac{1}{\sigma}} \frac{c}{2\sigma}\biggr)^{|\alpha|} \esp^{-\frac{c}{2\sigma}|\eps\xi|^{\frac{1}{\sigma}}}\biggr)^\sigma\biggl(\frac{c}{2\sigma}\biggr)^{-|\alpha|\sigma}\le \eps^{-|\alpha|}\biggl(\frac{c}{2\sigma}\biggr)^{-|\alpha|\sigma}(|\alpha|!)^\sigma\\
\le\eps^{-|\alpha|}\biggl(\frac{c}{2\sigma}\biggr)^{-|\alpha|\sigma} n^{|\alpha|\sigma}(\alpha!)^\sigma\le \eps^{-|\alpha|}c_\sigma^{|\alpha|}(\alpha !)^\sigma.
\end{multline}
Finally combining \eqref{est_b} with \eqref{est_c} we conclude that there exists a constant $C>0$ such that  
\[
|\partial^\alpha u_\eps(x)|\le C^{|\alpha|+1}(\alpha !)^\sigma \eps^{-|\alpha|}\eps^{-N-n},
\]
for all $x\in\R^n$ and $\eps\in(0,1]$.

(iv) If $(u_\eps)_\eps$ is a net of tempered distributions satisfying $(ii)$ then by calculations analogous to the ones above (replacing $-N$ with $q$) we have that
\[
|\partial^\alpha u_\eps(x)|\le C_q^{|\alpha|+1}(\alpha!)^\sigma \eps^{q-n},
\]
for all $\eps\in(0,1]$ and $x\in\R^n$.
\end{proof}


Making use of the previous definitions of $\gamma^\sigma$-moderate and negligible net (see Definition \eqref{def_negl_net} and the paragraph above),we introduce the quotient space
\[
\G^\sigma(\R^n):=\frac{\gamma^\sigma-\text{moderate\, nets}}{\gamma^\sigma-\text{negligible\, nets}}.
\]
We now investigate the relationship between $\G^\sigma(\R^n)$ and the classical Colombeau algebra 
\[
\G(\R^n)=\frac{\E_M(\R^n)}{\Neg(\R^n)}=\frac{C^\infty-\text{moderate\, nets}}{C^\infty-\text{negligible\, nets}}.
\]
We recall that a net $(u_\eps)_\eps$ is $C^\infty$-moderate is for all $K\Subset\R^n$ and all $\alpha\in\N^n$ there exist $c>0$ and $N\in\N$ such that
\beq
\label{mod_C_inf}
|\partial^\alpha u_\eps(x)|\le c\eps^{-N},
\eeq
for all $x\in K$ and $\eps\in(0,1]$. 
A net $(u_\eps)_\eps$ is $C^\infty$-negligible is for all $K\Subset\R^n$, all $\alpha\in\N^n$ and all $q\in\N$ there exists $c>0$ such that
\beq
\label{neg_C_inf}
|\partial^\alpha u_\eps(x)|\le c\eps^{q},
\eeq
uniformly in $x\in K$ and $\eps\in(0,1]$. 
For the general analysis of $\G(\R^n)$ we refer to e.g. Oberguggenberger
\cite{Oberguggenberger:Bk-1992}.

\begin{proposition}
\label{prop_Colombeau_embedd}
For all $\sigma\ge1$, 
\[
\G^\sigma(\R^n)\subseteq \G(\R^n).
\]
\end{proposition}
\begin{proof}
To prove that $\G^\sigma(\R^n)$ is a subalgebra of $\G(\R^n)$ we need to prove that $\gamma^\sigma$-moderate and $\gamma^\sigma$-negligible nets are elements of $\E_M(\R^n)$ and $\Neg(\R^n)$, respectively and that if a $\gamma^\sigma$-moderate net belongs to $\Neg(\R^n)$ then it is automatically $\gamma^\sigma$-negligible. The first two implications are clear from the definition of $\gamma^\sigma$-moderate and $\gamma^\sigma$-negligible net. Finally, if $(u_\eps)_\eps$ is $\gamma^\sigma$-moderate and belongs to $\Neg(\R^n)$ then for all $K\Subset\R^n$ we have
\[
|\partial^\alpha u_\eps(x)|^2=|\partial^\alpha u_\eps(x)||\partial^\alpha u_\eps(x)|\le c_K^{|\alpha|+1}(\alpha!)^\sigma \eps^{-N-|\alpha|}c_{K,q}\eps^q.
\]
Choosing $q=2q'+N$ and by simple estimates we get
\[
|\partial^\alpha u_\eps(x)|^2\le c_{K,q'}^{2|\alpha|+2}(\alpha!)^{2\sigma}\eps^{2q'-2|\alpha|},
\]
which implies that the net $(u_\eps)_\eps$ is $\gamma^\sigma$-negligible.
\end{proof}
The quotient space $\G^\sigma(\R^n)$ is a sheaf. This means that one can introduce a notion of restriction and a notion of support. More precisely, $x\in \R^n\setminus \supp\, u$ if there exists an open neighbourhood $V$ of $x$ such that $u|_{V}=0$ in $\G^\sigma(V)$.
Define $\G^\sigma_c(\R^n)$ as the algebra of compactly supported generalised functions in $\G^\sigma(\R^n)$. Making use of the previous arguments on $\gamma^\sigma$-moderate and -negligible nets we can prove the following proposition.
\begin{proposition}
\label{prop_34}
\leavevmode
\begin{itemize}
\item[(i)]
If $u\in \G^\sigma(\R^n)$ has compact support then it has a representative $(u_\eps)_\eps$ and a compact set $K$ such that $\supp\, u_\eps\subseteq K$ uniformly in $\eps$.  
\item[(ii)] $\gamma^\sigma_c(\R^n)$ is a subalgebra of $\G^\sigma_c(\R^n)$. 
\end{itemize} 
\end{proposition}
\begin{proof}
(i) We begin by recalling that if $u\in \G(\R^n)$ has compact support then it has a representative $(u_\eps)_\eps$ with $\supp\, u_\eps$ contained in a compact set $K$ uniformly with respect to $\eps$. In other words, there exists $\psi\in C^\infty_c(\R^n)$ identically one on a neighbourhood of $\supp\, u$ such that $\psi u=u$ in $\G(\R^n)$. It follows that if $u\in \G^\sigma(\R^n)$ has compact support then $\psi u=u$ in $\G^\sigma(\R^n)$. Indeed,
\[
|\partial^\alpha(\psi u_\eps)(x)|\le \sum_{\alpha'\le\alpha}\binom{\alpha}{\alpha'}|\partial^{\alpha'}\psi(x)||\partial^{\alpha-\alpha'}u_\eps(x)|\le c_\psi c_\psi^{|\alpha|+1}(\alpha !)^\sigma\eps^{-N-|\alpha|}.
\]
This means that $(\psi u_\eps)_\eps$ is $\gamma^\sigma_c$-moderate. Since $\psi u_\eps-u_\eps$ is $\gamma^\sigma$-moderate and belongs to $\Neg(\R^n)$ as well, we conclude that $(\psi u_\eps-u_\eps)_\eps$ is $\gamma^\sigma$-negligible. 
 
(ii) The inclusion $\gamma^\sigma_c(\R^n)\subseteq \G^\sigma_c(\R^n)$ is a straightforward consequence of the fact that if $u\in \gamma^\sigma_c(\R^n)$ then $(u-u\ast\varphi_\eps)_\eps$ is $\gamma^\sigma$-negligible by Proposition \ref{prop_reg_gevrey_sim} and $\supp [(u\ast\varphi_\eps)_\eps]=\supp\, u$.
\end{proof}
An analogous version of Proposition \ref{prop_34} can be proven for $\G^\sigma(\R^n)$ and $\gamma^\sigma(\R^n)$, but it goes beyond the purpose of this paper.

In this paper we will also make use of the following factor space.
\begin{definition}
\label{def_sol_space}
Let $(u_\eps(t,x))_\eps\in C^\infty([0,T]; \gamma^\sigma(\R^n))$. We say that the net $(u_\eps)_\eps$ is $C^\infty([0,T]; \gamma^\sigma(\R^n))$-moderate if for all $K\Subset\R^n$ there exist $N\in\N$, $c>0$ and, for all $k\in\N$ there exist $N_k>0$ and $c_k>0$ such that
\[
|\partial_t^k\partial^\alpha_x u_\eps(t,x)|\le c_k\eps^{-N_k} c^{|\alpha|+1}(\alpha !)^\sigma \eps^{-N-|\alpha|},
\]
for all $\alpha\in\N^n$, for all $t\in[0,T]$, $x\in K$ and $\eps\in(0,1]$.

We say that the net $(u_\eps)_\eps$ is $C^\infty([0,T]; \gamma^\sigma(\R^n))$-negligible if for all $K\Subset\R^n$, for all $k\in \N$ and for all $q\in\N$ there exists $c>0$ such that 
\[
|\partial_t^k\partial^\alpha_x u_\eps(t,x)|\le c^{|\alpha|+1}(\alpha!)^\sigma \eps^{q-|\alpha|},
\]
for all $\alpha\in\N^n$, for all $t\in[0,T]$, $x\in K$ and $\eps\in(0,1]$.

We denote the quotient space of $C^\infty([0,T]; \gamma^\sigma(\R^n))$-moderate nets with respect to $C^\infty([0,T]; \gamma^\sigma(\R^n))$-negligible nets by
\[
\G([0,T];\G^\sigma(\R^n)).
\]
\end{definition}
Note that the estimates in Definition \ref{def_sol_space} express the usual Colombeau properties in $t$ and the new Gevrey-Colombeau  features in $x$ and that
\[
\G^\sigma(\R^n)\subseteq \G([0,T];\G^\sigma(\R^n))\subseteq \G([0,T]\times\R^n).
\]
Moreover, in $\G([0,T];\G^\sigma(\R^n))$ one can make use, at the level of representatives, of the characterisations by Fourier transform seen above (uniformly in $t\in[0,T]$).

\subsection{Energy estimate and well-posedness}
\label{subsec_energy}
Let us define the energy
\[
E_{\delta,\eps}(t,\xi):=(Q^{(2)}_{\delta,\eps}(t,\xi) V(t,\xi), V(t,\xi)).
\]
We have
\begin{multline*}
\partial_t E_{\delta,\eps}(t,\xi)=(\partial_tQ^{(2)}_{\delta,\eps} V,V)+ i(Q^{(2)}_{\delta,\eps} D_tV,V)-i(Q^{(2)}_{\delta,\eps} V,D_tV)\\
=(\partial_tQ^{(2)}_{\delta,\eps} V,V)+i(Q^{(2)}_{\delta,\eps}A_{1,\eps}V,V)-i(Q^{(2)}_{\delta,\eps} V,A_{1,\eps}V)\\
=(\partial_tQ^{(2)}_{\delta,\eps} V,V)+i\lara{\xi}((Q^{(2)}_{\delta,\eps} A_\eps-A_\eps^\ast Q^{(2)}_{\delta,\eps})V,V).\\
\end{multline*}
It follows that
\begin{multline}
\label{EE}
\partial_t E_{\delta,\eps}(t,\xi)\le \frac{|(\partial_tQ^{(2)}_{\delta,\eps}(t,\xi) V(t,\xi),V(t,\xi))|E_{\delta,\eps}(t,\xi)}{(Q^{(2)}_{\delta,\eps}(t,\xi)V(t,\xi), V(t,\xi))}\\
+\lara{\xi}|((Q^{(2)}_{\delta,\eps} A_\eps-A_\eps^\ast Q^{(2)}_{\delta,\eps})(t,\xi)V(t,\xi),V(t,\xi))|.
\end{multline}
Let now 
\[
K_{\delta,\eps}(t,\xi):=\frac{|(\partial_tQ^{(2)}_{\delta,\eps}(t,\xi) V(t,\xi),V(t,\xi))|}{(Q^{(2)}_{\delta,\eps}(t,\xi)V(t,\xi), V(t,\xi))},
\]
provided that $V\neq 0$. Hence, we can rewrite \eqref{EE} as
\beq
\label{EE2}
\partial_t E_{\delta,\eps}(t,\xi)\le K_{\delta,\eps}(t,\xi)E_{\delta,\eps}(t,\xi)+\lara{\xi}|((Q^{(2)}_{\delta,\eps} A_\eps-A_\eps^\ast Q^{(2)}_{\delta,\eps})(t,\xi)V(t,\xi),V(t,\xi))|.
\eeq
By Proposition \ref{prop_qs_2}$(ii)$ we have that
\begin{multline*}
|((Q^{(2)}_{\delta,\eps} A_\eps-A_\eps^\ast Q^{(2)}_{\delta,\eps})(t,\xi)V(t,\xi),V(t,\xi))|\le C_2\delta (Q^{(2)}_{\delta,\eps})(t,\xi)V(t,\xi),V(t,\xi))\\
=C_2\delta E_{\delta,\eps}(t,\xi).
\end{multline*}
Hence
\beq
\label{EEfinal}
\partial_t E_{\delta,\eps}(t,\xi)\le (K_{\delta,\eps}(t,\xi)+C_2\delta \lara{\xi})E_{\delta,\eps}(t,\xi).
\eeq
In the following we take any fixed integer $k\ge 2$. Writing now 
\[
\int_0^T K_{\delta,\eps}(t,\xi)\, dt
\]
as
\[
\int_0^T\frac{|(\partial_tQ^{(2)}_{\delta,\eps}(t,\xi) V(t,\xi),V(t,\xi))|}{(Q^{(2)}_{\delta,\eps}(t,\xi) V(t,\xi), V(t,\xi))^{1-1/k}(Q^{(2)}_{\delta,\eps}(t,\xi) V(t,\xi), V(t,\xi))^{1/k}}\, dt,
\]
from the bound from below in Proposition \ref{prop_qs_2}$(i)$, Lemma \ref{lem_new_2} and the estimates on the roots $\lambda_{i,\eps}(t,\xi)$, $i=1,2$, we have that
\begin{multline}
\label{est_k}
\int_0^T K_{\delta,\eps}(t,\xi)\, dt\le \int_0^T \frac{|(\partial_tQ^{(2)}_{\delta,\eps}(t,\xi) V(t,\xi),V(t,\xi))|}{(Q^{(2)}_{\delta,\eps}(t,\xi) V(t,\xi), V(t,\xi))^{1-1/k}(C_2^{-1}\omega(\eps)^{2L}\delta^2|V(t,\xi)|^2)^{1/k}}\, dt\\
= C_2^{\frac{1}{k}}\delta^{-\frac{2}{k}}\omega(\eps)^{-\frac{2L}{k}} \int_0^T \frac{|(\partial_tQ^{(2)}_{\delta,\eps}(t,\xi) V(t,\xi),V(t,\xi))|}{(Q^{(2)}_{\delta,\eps}(t,\xi) V(t,\xi), V(t,\xi))^{1-1/k}|V(t,\xi)|^{2/k}}\, dt\\
\le C_2^{\frac{1}{k}}\omega(\eps)^{-\frac{2L}{k}}\delta^{-\frac{2}{k}} \Vert Q^{(2)}_{\delta,\eps}(\cdot,\xi)\Vert^{1/k}_{{C}^k([0,T])}
\le C_1\delta^{-\frac{2}{k}}\omega(\eps)^{-\frac{2L}{k}}\omega(\eps)^{-\frac{L}{k}-1},
\end{multline}
uniformly in all the variables and parameters. Combining now \eqref{est_k} with the estimate on $|((Q^{(2)}_{\delta,\eps} A_\eps-A_\eps^\ast Q^{(2)}_{\delta,\eps})(t,\xi)V(t,\xi),V(t,\xi))|$ above, by Gronwall lemma we obtain
\beq
\label{Gron_lem}
E_{\delta,\eps}(t,\xi)\le E_{\delta,\eps}(0,\xi)\esp^{C_1\delta^{-\frac{2}{k}}\omega(\eps)^{-\frac{3L}{k}-1}+C_2T\delta \lara{\xi}}\le  E_{\delta,\eps}(0,\xi)\esp^{C_T(\delta^{-\frac{2}{k}}\omega(\eps)^{-\frac{3L}{k}-1}+\delta \lara{\xi})}.
\eeq
As in \cite{GR:12} set $\delta^{-\frac{2}{k}}=\delta\lara{\xi}$. It follows that $\delta^{-\frac{2}{k}}=\lara{\xi}^{\frac{1}{\sigma}}$, where
\[
\sigma= 1+\frac{k}{2}.
\] 
Making use of the estimates in Proposition \ref{prop_qs_2}$(i)$, of the definition of $Q^{(2)}_{\delta,\eps}$ and of the fact that $\omega(\eps)^{-1}\ge 1$, we obtain
\begin{multline*}
C_2^{-1}\omega(\eps)^{2L}\delta^2|V(t,\xi)|^2\le E_{\delta,\eps}(t,\xi)\le E_{\delta,\eps}(0,\xi)\esp^{C_T\omega(\eps)^{-\frac{3L}{k}-1}\lara{\xi}^{\frac{1}{\sigma}}}\\
\le C_2\omega(\eps)^{-2L}|V(0,\xi)|^2\esp^{C_T\omega(\eps)^{-\frac{3L}{k}-1}\lara{\xi}^{\frac{1}{\sigma}}}.
\end{multline*}
This implies, for $M=(3L+k)/k$,
\begin{multline*}
|V(t,\xi)|^2\le C_2^2\delta^{-2}\omega(\eps)^{-4L}|V(0,\xi)|^2\esp^{C_T\omega(\eps)^{-M}\lara{\xi}^{\frac{1}{\sigma}}}\\
= C_2^2\omega(\eps)^{-4L}\lara{\xi}^{\frac{k}{\sigma}}|V(0,\xi)|^2\esp^{C_T\omega(\eps)^{-M}\lara{\xi}^{\frac{1}{\sigma}}},
\end{multline*}
or equivalently
\[
|V(t,\xi)|\le C\omega(\eps)^{-2L}\lara{\xi}^{\frac{k}{2\sigma}}|V(0,\xi)|\esp^{C\omega(\eps)^{-M}\lara{\xi}^{\frac{1}{\sigma}}},
\]
for a suitable constant $C>0$.

We begin by assuming that the initial data are in $\gamma^s(\R^n)$. This means that 
\[
|V(0,\xi)|\le C'_0\esp^{-C_0\lara{\xi}^{\frac{1}{s}}}.
\]
Since our solution is  depending on the parameter $\eps$ from now on we will adopt the notation $V_\eps$. Note that when the initial data are in $\gamma_c^s(\R^n)$ we do not need any regularisation to embed them in the algebra $\G^s(\R^n)$, due to Proposition \ref{prop_reg_gevrey_sim}(ii). Hence,
\beq
\label{EEV}
|V_\eps(t,\xi)|\le C\omega(\eps)^{-2L}\lara{\xi}^{\frac{k}{2\sigma}}|V_\eps(0,\xi)|\esp^{C\omega(\eps)^{-M}\lara{\xi}^{\frac{1}{\sigma}}},
\eeq
and by simple estimates
\begin{multline}
\label{est_V_eps}
|V_\eps(t,\xi)|\le C\omega(\eps)^{-2L}\lara{\xi}^{\frac{k}{2\sigma}}C'_0\esp^{-C_0\lara{\xi}^{\frac{1}{s}}}\esp^{C\omega(\eps)^{-M}\lara{\xi}^{\frac{1}{\sigma}}}\\
=CC'_0\omega(\eps)^{-2L}\lara{\xi}^{\frac{k}{2\sigma}}\esp^{-\frac{C_0}{2}\lara{\xi}^{\frac{1}{s}}}\esp^{-\frac{C_0}{2}\lara{\xi}^{\frac{1}{s}}+C\omega(\eps)^{-M}\lara{\xi}^{\frac{1}{\sigma}}}.
\end{multline}
If $s<\sigma$, the condition
\[
-\frac{C_0}{2}+C\omega(\eps)^{-M}\lara{\xi}^{\frac{1}{\sigma}-\frac{1}{s}}\le 0
\]
is equivalent to
\[
\begin{split}
C\omega(\eps)^{-M}\lara{\xi}^{\frac{1}{\sigma}-\frac{1}{s}} &\le \frac{C_0}{2},\\
\lara{\xi}^{\frac{1}{\sigma}-\frac{1}{s}}&\le \frac{C_0}{2}\frac{1}{C}\omega(\eps)^M,\\
\lara{\xi}^{\frac{1}{s}-\frac{1}{\sigma}}&\ge \big(\frac{C_0}{2}\frac{1}{C}\big)^{-1}\omega(\eps)^{-M},\\
\lara{\xi}&\ge  \biggl(\big(\frac{C_0}{2}\frac{1}{C}\big)^{-1}\omega(\eps)^{-M}\biggr)^{\frac{1}{\frac{1}{s}-\frac{1}{\sigma}}}
\end{split}
\]
or, in other words, to the condition
\beq
\label{R_eps}
\lara{\xi}\ge R_\eps:= \biggl(\big(\frac{C_0}{2}\frac{1}{C}\big)^{-1}\omega^{-M}(\eps)\biggr)^{\frac{1}{\frac{1}{s}-\frac{1}{\sigma}}}.
\eeq
Assume now that $\omega(\eps)^{-1}$ is moderate, i.e. $\omega(\eps)^{-1}\le c\eps^{-r}$ for some $r\ge 0$. Hence, there exists $N\in\N$ such that under the assumption \eqref{R_eps} the estimate \eqref{est_V_eps} yields
\beq
\label{est_prima}
|V_\eps(t,\xi)|\le c'\eps^{-N}\esp^{-C'\lara{\xi}^{\frac{1}{s}}},
\eeq
which proves that the net $U_\eps=\mathcal{F}^{-1}(V_\eps{1}_{\lara{\xi}\ge R_\eps})$ is $\gamma^s$-moderate. It remains to estimate $V_\eps(t,\xi)$ when $\lara{\xi}\le R_\eps$. Going back to \eqref{est_V_eps} we have that if $\lara{\xi}\le R_\eps$ then
\begin{multline*}
|V_\eps(t,\xi)|\le C\omega(\eps)^{-2L}\lara{\xi}^{\frac{k}{2\sigma}}C'_0\esp^{-C_0\lara{\xi}^{\frac{1}{s}}}\esp^{C\omega(\eps)^{-M}\lara{\xi}^{\frac{1}{\sigma}}}\\
\le
CC'_0\omega(\eps)^{-2L}\lara{\xi}^{\frac{k}{2\sigma}}\esp^{-\frac{C_0}{2}\lara{\xi}^{\frac{1}{s}}}\esp^{-\frac{C_0}{2}\lara{\xi}^{\frac{1}{s}}}\esp^{C\omega(\eps)^{-M}\lara{R_\eps}^{\frac{1}{\sigma}}}.
\end{multline*}
At this point, choosing $\omega(\eps)^{-M}\lara{R_\eps}^{\frac{1}{\sigma}}$ of logarithmic type, i.e.,
\begin{multline}
\label{log_scale_1}
\omega(\eps)^{-M}\lara{R_\eps}^{\frac{1}{\sigma}}\le c\log(\eps^{-1})\Leftrightarrow \omega(\eps)^{-M}{\omega(\eps)}^{\frac{-M\frac{1}{\sigma}}{\frac{1}{s}-\frac{1}{\sigma}}}\le c\log(\eps^{-1})\\ 
\Leftrightarrow \omega(\eps)^{-1}\le c (\log(\eps^{-1}))^{\frac{1}{M+\frac{M\frac{1}{\sigma}}{\frac{1}{s}-\frac{1}{\sigma}}}}\Leftrightarrow \omega(\eps)^{-1}\le c(\log(\eps^{-1}))^{\frac{\frac{1}{s}-\frac{1}{\sigma}}{\frac{1}{s}M}},
\end{multline}
we can conclude that there exists $N\in\N$ and $c',C'>0$ such that
\[
|V_\eps(t,\xi)|\le c'\esp^{-C'\lara{\xi}^{\frac{1}{s}}}\eps^{-N},
\]
for all $\eps\in(0,1]$, $t\in[0,T]$ and $\lara{\xi}\le R_\eps$. This  together with \eqref{est_prima} and Proposition \ref{prop_reg_gevrey_mod}(iii)  shows that the net $(U_\eps(t,\cdot))_\eps$ is $\gamma^s$-moderate on $\R^n$ for
\[
1< s<\sigma=1+\frac{k}{2}.
\]

We are now ready to state and prove the following well-posedness theorem.
\begin{theorem}
\label{theo_CP_1}
Let 
\[
\begin{split}
D_t^2u(t,x)-\sum_{i=1}^n b_i(t)D_tD_{x_i}u(t,x)-\sum_{i=1}^n a_i(t)D_{x_i}^2u(t,x)&=0,\\
u(0,x)&=g_0,\\
D_t u(0,x)&=g_1,\\
\end{split}
\]
where the coefficients $a_i$ and $b_i$ are real-valued distributions with compact support contained in $[0,T]$ and $a_i$ is non-negative for all $i=1,\dots,n$. Let $g_0$ and $g_1$ belong to $\gamma^s_c(\R^n)$ with $s>1$. Then there exists a suitable embedding of the coefficients $a_i$'s and $b_i$'s into $\G([0,T])$ such that he Cauchy problem above has a unique solution $u\in \G([0,T];\G^s(\R^n))$.  

\end{theorem}
\begin{proof}
We begin by writing the equation
\[
D_t^2u(t,x)-\sum_{i=1}^n b_i(t)D_tD_{x_i}u(t,x)-\sum_{i=1}^n a_i(t)D_{x_i}^2u(t,x)=0
\]
as an equation in $\G([0,T];\G^s(\R^n))$. This means that we replace the coefficients $a_i$ and $b_i$ with the equivalence classes of $(a_{i,\eps})_\eps$ and $(b_{i,\eps})_\eps$ in $\G([0,T])$ as in Section \ref{sec_quasi}. Since the initial data are in $\G^s(\R^n)$ they can be imbedded in $\G^s(\R^n)$ as they are, i.e. $[(g_0)]\in \G^s_c(\R^n)$ and $[(g_1)]\in \G^s_c(\R^n)$.

\emph{Existence.} We argue now at the level of representatives and we transform the equation to the first order system \eqref{system_hyp}. From the theory of weakly hyperbolic equations and in particular from \cite{GR:12,KS} we know that that the Cauchy problem
\[
D_t^2u(t,x)-\sum_{i=1}^n b_{i,\eps}(t)D_tD_{x_i}u(t,x)-\sum_{i=1}^n a_{i,\eps}(t)D_{x_i}^2u(t,x)=0
\]
with initial data $g_0, g_1\in \gamma^s_c(\R^n)$, has a net of (classical) solutions $(u_\eps)_\eps\in C^2([0,T]:\gamma^s(\R^n))$. More precisely, we know that given $s> 1$ and for $k\ge 2$ there exists a solution $(u_\eps)_\eps\in C^2([0,T]:\gamma^s(\R^n))$ provided that 
\[
1< s<1+\frac{k}{2}.
\]
So, in the arguments which follow we assume $s$ and $k$ in this relation and we perform the embedding of the coefficients $a_i$ and $b_i$ with a logarithmic scale of the type  $\omega^{-1}(\eps)= c(\log(\eps^{-1}))^{r}$, $c\ge 0$, as in \eqref{log_scale_1}, where $r$ depends on $s$ and $k$.

It is our task to show that this net is moderate. From the energy estimates in Subsection \ref{subsec_energy} at the Fourier transform level we have that the net $(u_\eps)_\eps$ (or better the corresponding $(U_\eps)_\eps$) is $\gamma^s(\R^n)$-moderate with respect to $x$ with $s$ as above. Since this moderateness estimate is uniform in $t$ and the coefficients of the equation are smooth and moderate in $t\in[0,T]$ as well, by induction on the $t$-derivatives and arguing as in \cite{LO:91} we can easily conclude that $(u_\eps)_\eps$ is $C^\infty([0,T]; \gamma^\sigma(\R^n))$-moderate for 
\[
1< s<1+\frac{k}{2}.
\]
Hence, $(u_\eps)_\eps$ generates a solution $u\in \G([0,T];\G^s(\R^n))$ to our Cauchy problem.

\emph{Uniqueness.} Assume now that the Cauchy problem has another solution $v\in \G([0,T];\G^s(\R^n))$. At the level of representatives this means
\[
D_t^2(u_\eps-v_\eps)(t,x)-\sum_{i=1}^n b_{i,\eps}(t)D_tD_{x_i}(u_\eps-v_\eps)(t,x)-\sum_{i=1}^n a_{i,\eps}(t)D_{x_i}^2(u_\eps-v_\eps)(t,x)=f_\eps(t,x),
\]
with initial data
\[
\begin{split}
u_\eps(0,x)-v_\eps(0,x)&=n_{0,\eps}(x),\\
D_t u_\eps(0,x)-D_t v_\eps(0,x)&=n_{1,\eps}(x),
\end{split}
\]
where $(f_\eps)_\eps$ is $C^\infty([0,T];\gamma^s(\R^n))$-negligible and $(n_{0,\eps})_\eps$ and $(n_{1,\eps})_\eps$ are both compactly supported and $\gamma^s(\R^n)$-negligible. The corresponding first order system is 
\[
D_t\left(
                             \begin{array}{c}
                               w_{1,\eps} \\
                                w_{2,\eps} \\
                             \end{array}
                           \right)
= \left(
    \begin{array}{cc}
      0 & \lara{D_x}\\
      \sum_{i=1}^n a_{i,\eps}(t)D_{x_i}^2\lara{D_x}^{-1}& \sum_{i=1}^n b_{i,\eps}(t)D_{x_i} \\
           \end{array}
  \right)
  \left(\begin{array}{c}
                               w_{1,\eps} \\
                               
                               w_{2,\eps} \\
                             \end{array}
                           \right)+\left(\begin{array}{c}
                               0 \\
                               
                               f_\eps \\
                             \end{array}
                           \right),
                           \]
                           where $w_{1,\eps}$ and $w_{2,\eps}$ are obtained via the transformation
                           \[
                           w_{j,\eps}=D_t^{j-1}\lara{D_x}^{2-j}(u_\eps-v_\eps),\quad j=1,2.
                           \]
                                
This system will be studied after Fourier transform, as a system of the type
\[
D_t V_\eps(t,\xi)=\lara{\xi}A_{\eps}(t,\xi)V+F_\eps,
\]
with
\[
F_\eps=\left(\begin{array}{c}
                              0 \\
                                
                              \mathcal{F}_{x\to\xi}f_\eps \\
                             \end{array}
                           \right),
\]
These kind of systems and the corresponding weakly hyperbolic equations (with right hand-side) have been investigated in \cite{GR:11} under even less regular assumptions on the coefficients (H\"older). In particular, see Theorem 3 in \cite{GR:11}, Gevrey well-posedness results have been obtained for 
\[
1< s<1+\frac{k}{2}.
\]
The proof of Theorem 3 in \cite{GR:11} can be easily adapted to our situation by inserting everywhere a multiplicative factor $\omega(\eps)^{-L}$ coming from the regularisation of the coefficients and by replacing $\esp^{-\rho(t)\lara{\xi}^{\frac{1}{s}}}$ with  $\esp^{-\rho(t)\eps^{\frac{1}{s}}\lara{\xi}^{\frac{1}{s}}}$ in the formula (4.1) defining $V$ in \cite{GR:11}. The estimate (4.9) in \cite{GR:11} is therefore transformed into
\beq
\label{est_uniqueness}
|V_\eps(t,\xi)|\le c_1\omega(\eps)^{-N}\lara{\xi}^N\esp^{\kappa_1 \eps^{\frac{1}{s}}\lara{\xi}^{\frac{1}{s}}}|V_\eps(0,\xi)|+c_2\omega(\eps)^{-N}\lara{\xi}^N\esp^{\kappa_2\eps^{\frac{1}{s}}\lara{\xi}^{\frac{1}{s}}}|\widehat{F_\eps}(t,\xi)|,
\eeq
where $N\in\N$ depends on the equation or better on the regularity of the coefficients and $\kappa_1,\kappa_2>0$ can be chosen small enough. It follows that since the initial data $V_\eps(0,\xi)$ and the right-hand side $F_\eps(t,\xi)$ are negligible then $(V_\eps)_\eps$ is negligible as well in the suitable function spaces, or in other words, $(u_\eps-v_\eps)_\eps$ is $C([0,T], \gamma^s(\R^n))$-negligible. From the equation itself and the fact that the coefficients are nets of smooth functions one can deduce that the net $(u_\eps-v_\eps)_\eps$ is smooth in $t$ as well and more precisely that it is  $C^\infty([0,T], \gamma^s(\R^n))$-negligible. This proves that $u=v$ in $\G([0,T];\G^s(\R^n))$.

\end{proof}

\vspace{0.2cm}

\section{Case 2: well-posedness for smooth initial data }
We now work under the assumption that the initial data $g_0$ and $g_1$ are not Gevrey but still smooth. More precisely, $g_0, g_1\in C^\infty_c(\R^n)$. By convolution with a mollifier $\varphi_\eps$ as in Case 1 we get a net of smooth functions. It is our aim to find for a function $u\in C^\infty_c(\R^n)$ a new regularisation of the type
\[
u\ast\rho_\eps
\]
such that the corresponding net is Gevrey. This will allow us to embed the initial data $g_0$ and $g_1$ in an algebra of Gevrey-Colombeau type and to proceed with the well-posedness of the Cauchy problem \eqref{intro_CP}.

We begin with the following regularisation inspired by \cite{BenBou:09}.

\subsection{Gevrey regularisation of smooth functions with compact support}\hspace{1cm}

In the sequel $\mathcal{S}^{(\sigma)}(\R^n)$, $\sigma>1$, denotes the space of all $\varphi\in\Cinf(\R^n)$ such that
\[
\Vert \varphi\Vert_{b,\sigma}=\sup_{\alpha,\beta\in\N^n}\int_{\R^n}\frac{|x^\beta|}{b^{|\alpha+\beta|}\alpha !^\sigma \beta !^\sigma}|\partial^\alpha \varphi(x)|\, dx<\infty
\]
for all $b>0$.

We recall that the Gelfand-Shilov space $\mathcal{S}^{(\sigma)}(\R^n)$ is Fourier transform invariant (see e.g. \cite[Chapter 6]{NiRo:10} and \cite{Teo:06}). It follows that taking the inverse Fourier transform $\phi=\mF^{-1}\psi$ of a function $\psi\in\mathcal{S}^{(\sigma)}(\R^n)$ identically $1$ in a neighborhood of $0$ one gets a function $\phi\in\mathcal{S}^{(\sigma)}(\R^n)$ with
\beq
\label{mollifier}
\int\phi(x)\, dx=1,\qquad \text{and}\qquad \int x^\alpha\phi(x)\, dx=0,\quad \text{for all $\alpha\neq 0$.}
\eeq
For instance, one can take $\psi\in \gamma^{(\sigma)}(\R^n)\cap \Cinfc(\R^n)$, where $\gamma^{(\sigma)}(\R^n)$ is the space of all $f\in\Cinf(\R^n)$ such that for all compact subset $K$ of $\R^n$ and all $b>0$ there exists $c>0$ such that $\sup_{x\in K}|\partial^\alpha f(x)|\le c\, b^{|\alpha|}\alpha !^\sigma$ for all $\alpha\in\N^n$.

We say that $\phi\in\mathcal{S}^{(\sigma)}(\R^n)$ is a mollifier if the property \eqref{mollifier} holds. Let now $\chi\in \gamma^\sigma(\R^n)$ with $0\le\chi\le 1$, $\chi(x)=0$ for $|x|\ge 2$ and $\chi(x)=1$ for $|x|\le 2$. We define (as in \cite{BenBou:09}) the net of Gevrey functions
\beq
\label{def_rho}
\rho_\eps(x):=\eps^{-n}\phi\biggl(\frac{x}{\eps}\biggr)\chi(x|\log\eps|).
\eeq
Note that the following estimates are valid for $\eps$ small enough, i.e., for all $\eps\in(0,\eta]$ with $\eta\in(0,1]$. Without loss of generality we can assume $\eta=1$.
\begin{proposition}
\label{prop_reg_smooth}
Let $u\in C^\infty_c(\R^n)$ and $\rho_\eps$ as above. Then, there exists $K\Subset\R^n$ such that $\supp(u\ast\rho_\eps)\subseteq K$ for all $\eps$ small enough and 
\begin{itemize}
\item[(i)] there exists $c>0$ and $\eta\in(0,1]$ such that
\[
|\partial^\alpha(u\ast\rho_\eps)(x)|\le c^{|\alpha|+1} (\alpha !)^\sigma \eps^{-|\alpha|}
\]
for all $\alpha\in\N^n$, $x\in\R^n$ and $\eps\in(0,\eta]$, or in other words, $(u\ast\rho_\eps)_\eps$ is $\gamma^\sigma_c$-moderate.
\item[(ii)] The net $(u\ast\rho_\eps-u)_\eps$ is compactly supported uniformly in $\eps$ and $C^\infty$-negligible.
\item[(iii)] There exist $c,c'>0$ and $\eta\in(0,1]$ such that
\[
|\widehat{u_\eps}(\xi)|\le c'\,\esp^{-c\,\eps^{\frac{1}{\sigma}}\lara{\xi}^{\frac{1}{\sigma}}},
\]
for all $\xi\in\R^n$ and $\eps\in(0,\eta]$.
\end{itemize}
\end{proposition}
\begin{proof}
(i) We begin by observing that there exists a compact set $K\subseteq\R^n$ such that $\supp(u\ast\rho_\eps)\subseteq K$ for all $\eps\in(0,1/2]$. Indeed, since the function $u$ has compact support and $\supp\,\rho_\eps\subseteq|\log\eps|^{-1} (\supp\,\chi)$ we get the inclusion
\[
\supp(u\ast\rho_\eps)\subseteq \supp\, u + |\log(1/2)|^{-1}\supp\,\chi.
\]
We write $\partial^\alpha(u\ast\rho_\eps)(x)$ as
\[
(u\ast\partial^\alpha\rho_\eps)(x)=\eps^{-n}\sum_{\gamma\le\alpha}\binom{\alpha}{\gamma}\partial^\gamma\phi\biggl(\frac{x}{\eps}\biggr)\eps^{-|\gamma|}\partial^{\alpha-\gamma}\chi(x|\log\eps|)|\log\eps|^{|\alpha-\gamma|}.
\]
Hence, the change of variable $y/\eps=z$ in 
\[
\int_{\R^n} u(x-y)\partial^\gamma\phi\biggl(\frac{y}{\eps}\biggr)\partial^{\alpha-\gamma}\chi(y|\log\eps|)\, dy
\]
entails
\begin{multline}
\label{conv_1}
|\partial^\alpha(u\ast\rho_\eps)(x)|\\
\le \sum_{\gamma\le\alpha}\binom{\alpha}{\gamma}\eps^{-|\gamma|}|\log\eps|^{|\alpha-\gamma|}\int_{\R^n}|u(x-\eps z)||\partial^\gamma\phi(z)||\partial^{\alpha-\gamma}\chi(\eps|\log\eps|z)|\, dz.
\end{multline}
Since $\chi\in \gamma^\sigma(\R^n)$ is compactly supported, there exists a constant $c_\varphi>0$ such that 
\beq
\label{gevrey_v}
|\partial^{\alpha-\gamma}\chi(\eps|\log\eps|z)|\le c_\chi^{|\alpha-\gamma|+1}(\alpha-\gamma)!^\sigma,
\eeq
for all $z\in\R^n$ and $\eps\in(0,1/2]$. Thus, combining \eqref{conv_1} with \eqref{gevrey_v} we obtain the estimate
\begin{multline}
\label{conv_2}
|\partial^\alpha(u\ast\rho_\eps)(x)|\le \sum_{\gamma\le\alpha}\binom{\alpha}{\gamma}\eps^{-|\gamma|}|\log\eps|^{|\alpha-\gamma|} c_\chi^{|\alpha-\gamma|+1}(\alpha-\gamma)!^\sigma\\
\int_{\R^n}\frac{|u(x-\eps z)||\partial^\gamma\phi(z)|(\gamma!)^\sigma}{(\gamma!)^\sigma}\, dz\\
\le c(u,\chi)\sum_{\gamma\le\alpha}\binom{\alpha}{\gamma}\eps^{-|\gamma|}|\log\eps|^{|\alpha-\gamma|} c_\chi^{|\alpha-\gamma|}(\alpha-\gamma)!^\sigma\Vert\phi\Vert_{\sigma,1}\gamma!^\sigma\\
\le c(u,\chi,\phi)\sum_{\gamma\le\alpha}\binom{\alpha}{\gamma}\eps^{-|\gamma|}|\log\eps|^{|\alpha-\gamma|}c_\chi^{|\alpha-\gamma|}(\alpha-\gamma)!^\sigma\gamma!^\sigma.
\end{multline}
Since $|\log \eps|$ is bounded by $\eps^{-1}$, $\sum_{\gamma\le\alpha}\binom{\alpha}{\gamma}=2^{|\alpha|}$ and $\delta!\le|\delta|!\le |\delta|^{|\delta|}$ for all $\delta\in\N^n$ we can conclude from \eqref{conv_2} that
\begin{multline*}
|\partial^\alpha(u\ast\rho_\eps)(x)|\le c\,c_1^{|\alpha|}\eps^{-|\alpha|}\sum_{\gamma\le\alpha}\binom{\alpha}{\gamma}|\alpha-\gamma|^{\sigma|\alpha-\gamma|}|\gamma|^{\sigma|\gamma|}
\le c\, c_1^{|\alpha|}\eps^{-|\alpha|}2^{|\alpha|}|\alpha|^{\sigma|\alpha|}\\
\le c' c_1^{|\alpha|}\eps^{-|\alpha|} 2^{|\alpha|}\espo^{\sigma|\alpha|}(\alpha!)^\sigma.
\end{multline*}
At this point collecting the terms with exponent $|\alpha|$ we conclude that there exist a constants $C>0$ and $C_1>0$ such that 
\beq
\label{final_est}
|\partial^\alpha(u\ast\rho_\eps)(x)|\le C^{|\alpha|}(\alpha!)^\sigma \eps^{-|\alpha|},
\eeq
uniformly in $\eps\in(0,1/2]$. 

(ii) By embedding of $C^\infty_c(\R^n)$ into the Colombeau algebra $\G(\R^n)$ we know that the net $(u-u\ast\phi_\eps)_\eps$ is $C^\infty$-negligible. It is easy to check that $(u\ast\phi_\eps-u\ast\rho_\eps)_\eps$ is $C^\infty$-negligible as well. Hence, $(u-u\ast\rho_\eps)_\eps$ is $C^\infty$-negligible.

(iii) In (i) we have proven that the net $(u\ast\rho_\eps)_\eps$ is $\gamma^\sigma_c$-moderate and has support contained in a compact set $K$ uniformly with respect to $\eps$. So, by Proposition \ref{prop_reg_gevrey_mod}(i) we immediately conclude that there exist $c,c'>0$ such that
\[
|\widehat{u_\eps}(\xi)|\le c'\,\esp^{-c\,\eps^{\frac{1}{\sigma}}\lara{\xi}^{\frac{1}{\sigma}}},
\]
for all $\xi\in\R^n$ and $\eps\in(0,1/2]$.
\end{proof}
In the sequel $\iota$ denotes the map
\[
C^\infty_c(\R^n)\to \G^\sigma_c(\R^n): u\mapsto [(u\ast \rho_\eps)_\eps].
\]
\begin{proposition}
\label{prop_embedd_Cinf}
\leavevmode
\begin{itemize}
\item[(i)] The map $\iota$ is injective on $C^\infty_c(\R^n)$.
\item[(ii)] If $u\in \gamma^\sigma_c(\R^n)$ then $(u\ast\phi_\eps- u\ast\rho_\eps)_\eps$ is $\gamma^\sigma$-negligible.
\end{itemize}
\end{proposition}
\begin{proof}
(i) Let $u\in C^\infty_c(\R^n)$ and $(u\ast\rho_\eps)_\eps$ be $\gamma^\sigma_c$-negligible. Hence, $(u\ast\rho_\eps)_\eps$ is $C^\infty$-negligible. Since $(u\ast\phi_\eps-u\ast\rho_\eps)_\eps$ is $C^\infty$-negligible we conclude that the net $(u\ast\phi_\eps)_\eps$ is $C^\infty$-negligible. By embedding of $C^\infty(\R^n)$ into the Colombeau algebra $\G(\R^n)$ it follows that $u=0$. This shows that the map $\iota$ is injective.


(ii) We write $(u\ast\phi_\eps- u\ast\rho_\eps)(x)$ as
\[
\int_{\R^n} u(x-\eps y)\phi(y)(1-\chi(y\eps|\log(\eps)|)\, dy.
\]
Hence, by the properties of $\chi$ and by the vanishing moments of $\varphi$, for any integer $q>1$ we get
\begin{multline*}
|\partial^\alpha(u\ast\phi_\eps- u\ast\rho_\eps)(x)|\\
\le \int_{\R^n}|\partial^\alpha u(x-\eps y)||\phi(y)|\sum_{|\beta|=q}\frac{|\partial^\beta\chi (y\eps|\log\eps|\theta)|}{\beta !}|\eps\log(\eps)y|^{|\beta|}\, dy\\
\le c^{|\alpha|+1} (\alpha!)^\sigma \int_{\R^n}|\phi(y)||y|^q c(\chi,q)|\eps\log(\eps)|^q\, dy\\
\le c(q,\chi,\phi) c^{|\alpha|+1} (\alpha!)^\sigma \eps^{\frac{q}{2}}.
\end{multline*}
This proves that the net $(u\ast\phi_\eps- u\ast\rho_\eps)(x)$ is $\gamma^\sigma$-negligible.
\end{proof}
Concluding, we can state that the algebra $\G^\sigma_c(\R^n)$ contains not only $\gamma^\sigma_c(\R^n)$ but also $C^\infty_c(\R^n)$ as a subalgebra. This is obtained by modifying the embedding from $u\ast\varphi_\eps$ in Section \ref{sec_case1} to $u\ast\rho_\eps$.

\subsection{Energy estimates and well-posedness}
\label{subsec_EE_2}
We now take initial data $g_0, g_1$ in $C^\infty_c(\R^n)$ and we embed them in $\G^s_c(\R^n)$ as $g_0\ast\rho_\eps$ and $g_1\ast\rho_\eps$. By repeating the transformation into first order system and the energy estimates of Case 1 at the Fourier transform level we arrive at \eqref{EEV}, i.e.
\[
|V_\eps(t,\xi)|\le C\omega(\eps)^{-2L}\lara{\xi}^{\frac{k}{2\sigma}}|V_\eps(0,\xi)|\esp^{C\omega(\eps)^{-M}\lara{\xi}^{\frac{1}{\sigma}}},
\]
for a suitable constant $C>0$ and $M=(3L+k)/k$. Since
\[
|V_\eps(0,\xi)|\le C'_0\esp^{-C_0\eps^{\frac{1}{s}}\lara{\xi}^{\frac{1}{s}}},
\]
we get
\begin{multline}
\label{est_V_eps_2}
|V_\eps(t,\xi)|\le C\omega(\eps)^{-2L}\lara{\xi}^{\frac{k}{2\sigma}}C'_0\esp^{-C_0\eps^{\frac{1}{s}}\lara{\xi}^{\frac{1}{s}}}\esp^{C\omega(\eps)^{-M}\lara{\xi}^{\frac{1}{\sigma}}}\\
=CC'_0\omega(\eps)^{-2L}\lara{\xi}^{\frac{k}{2\sigma}}\esp^{-\frac{C_0}{2}\eps^{\frac{1}{s}}\lara{\xi}^{\frac{1}{s}}}\esp^{-\frac{C_0}{2}\eps^{\frac{1}{s}}\lara{\xi}^{\frac{1}{s}}+C\omega(\eps)^{-M}\lara{\xi}^{\frac{1}{\sigma}}}.
\end{multline}
Recall that $s>1$ and that $k$ is any fixed integer with $k\ge 2$. Now, if $s<\sigma$, the following inequalities are equivalent:
\[
\begin{split}
-\frac{C_0}{2}\eps^{\frac{1}{s}}+C\omega(\eps)^{-M}\lara{\xi}^{\frac{1}{\sigma}-\frac{1}{s}}&\le 0,\\
C\omega(\eps)^{-M}\lara{\xi}^{\frac{1}{\sigma}-\frac{1}{s}} &\le \frac{C_0}{2}\eps^{\frac{1}{s}},\\
\lara{\xi}^{\frac{1}{\sigma}-\frac{1}{s}}&\le \frac{C_0}{2}\frac{1}{C}\omega(\eps)^{M}\eps^{\frac{1}{s}},\\
\lara{\xi}^{\frac{1}{s}-\frac{1}{\sigma}}&\ge \biggl(\big(\frac{C_0}{2}\frac{1}{C}\big)^{-1}\omega(\eps)^{-M}\eps^{-\frac{1}{s}}\biggr),\\
\lara{\xi}&\ge  \biggl(\big(\frac{C_0}{2}\frac{1}{C}\big)^{-1}\omega(\eps)^{-M}\eps^{-\frac{1}{s}}\biggr)^{\frac{1}{\frac{1}{s}-\frac{1}{\sigma}}},\\
\end{split}
\]
or, in other words,
\beq
\label{R_eps_2}
\lara{\xi}\ge R_\eps:= \biggl(\big(\frac{C_0}{2}\frac{1}{C}\big)^{-1}\omega(\eps)^{-M}\eps^{-\frac{1}{s}}\biggr)^{\frac{1}{\frac{1}{s}-\frac{1}{\sigma}}}.
\eeq
As in the previous case we take $\omega(\eps)^{-1}$ moderate. Under the assumption \eqref{R_eps_2} the estimate \eqref{est_V_eps_2} implies, for some $N\in\N$,
\beq
\label{est_prima_2}
|V_\eps(t,\xi)|\le C'\eps^{-N}\esp^{-C'\eps^{\frac{1}{s}}\lara{\xi}^{\frac{1}{s}}}.
\eeq
This shows that the net $U_\eps=\mathcal{F}^{-1}(V_\eps 1_{\lara{\xi}\ge R_\eps})$ is $\gamma^s$-moderate. We still have to estimate $V_\eps(t,\xi)$ when $\lara{\xi}\le R_\eps$. Going back to \eqref{est_V_eps_2} we have that if $\lara{\xi}\le R_\eps$ then
\begin{multline}
\label{est_V_2_3}
|V_\eps(t,\xi)|\le C\omega(\eps)^{-2L}\lara{\xi}^{\frac{k}{2\sigma}}C'_0\esp^{-C_0\eps^{\frac{1}{s}}\lara{\xi}^{\frac{1}{s}}}\esp^{C\omega(\eps)^{-M}\lara{\xi}^{\frac{1}{\sigma}}}\\
\le CC'_0\omega(\eps)^{-2L}\lara{\xi}^{\frac{k}{2\sigma}}\esp^{-\frac{C_0}{2}\eps^{\frac{1}{s}}\lara{\xi}^{\frac{1}{s}}}\esp^{-\frac{C_0}{2}\eps^{\frac{1}{s}}\lara{\xi}^{\frac{1}{s}}+C\omega(\eps)^{-M}\lara{R_\eps}^{\frac{1}{\sigma}}}
\end{multline}
At this point, choosing $\omega^{-M}(\eps)\lara{R_\eps}^{\frac{1}{\sigma}}$ of logarithmic type, i.e.,
\begin{multline}
\label{log_scale_2}
\omega^{-M}(\eps)\lara{R_\eps}^{\frac{1}{\sigma}}\le c\log(\eps^{-1})\Leftrightarrow \omega(\eps)^{-M}{\omega(\eps)}^{\frac{-M\frac{1}{\sigma}}{\frac{1}{s}-\frac{1}{\sigma}}}\eps^{-\frac{1}{s}\frac{\frac{1}{\sigma}}{\frac{1}{s}-\frac{1}{\sigma}}}\le c\log(\eps^{-1})\\ 
\Leftrightarrow \omega(\eps)^{-1}\le c (\log(\eps^{-1}))^{\frac{1}{M+\frac{M\frac{1}{\sigma}}{\frac{1}{s}-\frac{1}{\sigma}}}}\eps^{\frac{1}{\sigma M}},
\Leftrightarrow \omega(\eps)^{-1}\le c (\log(\eps^{-1}))^{\frac{\frac{1}{s}-\frac{1}{\sigma}}{\frac{1}{s}M}}\eps^{\frac{1}{\sigma M}},
\end{multline}
we can conclude that there exists $N\in\N$ and $c',C'>0$ such that
\[
|V_\eps(t,\xi)|\le c'\esp^{-C'\eps^{\frac{1}{s}}\lara{\xi}^{\frac{1}{s}}}\eps^{-N},
\]
for all $\eps\in(0,1]$, $t\in[0,T]$ and $\lara{\xi}\le R_\eps$.  Combing this last estimate with \eqref{est_prima_2} we can conclude, by Proposition \ref{prop_reg_gevrey_mod}$(iii)$, that, as in the previous case, the net $(U_\eps(t,\cdot))_\eps$ is $\gamma^s$-moderate on $\R^n$ for
\[
1< s<\sigma=1+\frac{k}{2}.
\]
We are now ready to state the following well-posedness theorem.
\begin{theorem}
\label{theo_CP_2}
Let 
\[
\begin{split}
D_t^2u(t,x)-\sum_{i=1}^n b_i(t)D_tD_{x_i}u(t,x)-\sum_{i=1}^n a_i(t)D_{x_i}^2u(t,x)&=0,\\
u(0,x)&=g_0,\\
D_t u(0,x)&=g_1,\\
\end{split}
\]
where the coefficients $a_i$ and $b_i$ are real valued distributions with compact support contained in $[0,T]$ and $a_i$ is non-negative for all $i=1,\dots,n$. Let $g_0$ and $g_1$ belong to $C^\infty_c(\R^n)$. Then, for all $s>1$ there exists a suitable embedding of the coefficients $a_i$ and $b_i$ into $\G([0,T])$ such that the Cauchy problem above has a unique solution $u\in \G([0,T];\G^s(\R^n))$.

\end{theorem}
\begin{proof}
We reduce the Cauchy problem above to a first order system and embed of the coefficients and initial data in the corresponding Colombeau algebras ($a_{i,\eps}\ast\psi_{\omega(\eps)}$, $b_{i,\eps}\ast\psi_{\omega(\eps)}$, $g_0\ast\rho_\eps$, $g_1\ast\rho_\eps$). Note that, we embed the coefficients $a_i$ and $b_i$ by means of a net $\omega(\eps)$ with $\omega^{-1}(\eps)\le c\, \eps^{r_1}(\log(\eps^{-1}))^{r_2}$ as in \eqref{log_scale_2}, where $r_1$ and $r_2$ depend on $s$ and fixed $k\ge 2$ with $1\le s<1+k$.

The energy estimates of Subsection \ref{subsec_EE_2} and the same arguments of Case 1 show the existence of a solution $u\in \G([0,T];\G^s(\R^n))$. The uniqueness of the solution $u$ is obtained as in the proof of Theorem \ref{theo_CP_1}.
\end{proof}

\section{Case 3: well-posedness for distributional initial data }
We pass now to consider distributional initial data, i.e. $g_0,g_1\in \E'(\R^n)$, and to investigate their convolution with the mollifier $\rho_\eps$.

\begin{proposition}
\label{prop_reg_distr}
Let $u\in\E'(\R^n)$ and $\rho_\eps$ as in \eqref{def_rho} . Then, there exists $K\Subset\R^n$ such that $\supp(u\ast\rho_\eps)\subseteq K$ for all $\eps$ small enough and there exist $C>0$, $N\in\N$ and $\eta\in(0,1]$ such that
\[
|\partial^\alpha(u\ast\rho_\eps)(x)|\le C^{|\alpha|+1} (\alpha !)^\sigma \eps^{-|\alpha|-N}
\]
for all $\alpha\in\N^n$, $x\in\R^n$ and $\eps\in(0,\eta]$.
\end{proposition}
\begin{proof}
The following proof differs from the proofs of Propositions \ref{prop_reg_gevrey_sim} and \ref{prop_reg_smooth} in terms of mollifier and dependence in $\eps$. We begin by noting that there exists a compact set $K\subseteq\R^n$ such that $\supp(u\ast\rho_\eps)\subseteq K$ for all $\eps\in(0,1/2]$. Indeed, since the distribution $u$ has compact support and $\supp\,\rho_\eps\subseteq|\log\eps|^{-1} (\supp\,\chi)$ we get the inclusion
\[
\supp(u\ast\rho_\eps)\subseteq \supp\, u + |\log(1/2)|^{-1}\supp\,\chi.
\]
By the structure of $u$ we know that there exists a continuous and compactly supported function $g$ such that 
\[
\partial^\alpha(u\ast\rho_\eps)=\partial^\alpha(\partial^\beta g\ast\rho_\eps)=g\ast\partial^{\alpha+\beta}\rho_\eps,
\]
where
\[
\partial^{\alpha+\beta}\rho_\eps=\eps^{-n}\sum_{\gamma\le\alpha+\beta}\binom{\alpha+\beta}{\gamma}\partial^\gamma\phi\biggl(\frac{x}{\eps}\biggr)\eps^{-|\gamma|}\partial^{\alpha+\beta-\gamma}\chi(x|\log\eps|)|\log\eps|^{|\alpha+\beta+\gamma|}.
\]
Hence, the change of variable $y/\eps=z$ in 
\[
\int_{\R^n} g(x-y)\partial^\gamma\phi\biggl(\frac{y}{\eps}\biggr)\partial^{\alpha+\beta-\gamma}\chi(y|\log\eps|)\, dy
\]
entails
\begin{multline}
\label{conv_1_b}
|\partial^\alpha(u\ast\rho_\eps)(x)|\le \sum_{\gamma\le\alpha+\beta}\binom{\alpha+\beta}{\gamma}\eps^{-|\gamma|}|\log\eps|^{|\alpha+\beta+\gamma|}\\
\int_{\R^n}|g(x-\eps z)||\partial^\gamma\phi(z)||\partial^{\alpha+\beta-\gamma}\chi(\eps|\log\eps|z)|\, dz.
\end{multline}
Since $\chi\in \gamma^\sigma(\R^n)$ is compactly supported, there exists a constant $c_\chi>0$ such that 
\beq
\label{gevrey_v_b}
|\partial^{\alpha+\beta-\gamma}\chi(\eps|\log\eps|z)|\le c_\varphi^{|\alpha+\beta-\gamma|+1}(\alpha+\beta-\gamma)!^\sigma,
\eeq
for all $z\in\R^n$ and $\eps\in(0,1/2]$. Hence, combining \eqref{conv_1_b} with \eqref{gevrey_v_b} we obtain the estimate
\begin{multline}
\label{conv_2_b}
|\partial^\alpha(u\ast\rho_\eps)(x)|\le \sum_{\gamma\le\alpha+\beta}\binom{\alpha+\beta}{\gamma}\eps^{-|\gamma|}|\log\eps|^{|\alpha+\beta+\gamma|} c_\chi^{|\alpha+\beta-\gamma|+1}(\alpha+\beta-\gamma)!^\sigma\\
\int_{\R^n}\frac{|g(x-\eps z)||\partial^\gamma\phi(z)|(\gamma!)^\sigma}{(\gamma!)^\sigma}\, dz\\
\le c(g,\chi)\sum_{\gamma\le\alpha+\beta}\binom{\alpha+\beta}{\gamma}\eps^{-|\gamma|}|\log\eps|^{|\alpha+\beta+\gamma|} c_\chi^{|\alpha+\beta-\gamma|}(\alpha+\beta-\gamma)!^\sigma\Vert\phi\Vert_{\sigma,1}\gamma!^\sigma\\
\le c(g,\chi,\phi)\sum_{\gamma\le\alpha+\beta}\binom{\alpha+\beta}{\gamma}\eps^{-|\gamma|}|\log\eps|^{|\alpha+\beta+\gamma|}c_\chi^{|\alpha+\beta-\gamma|}(\alpha+\beta-\gamma)!^\sigma\gamma!^\sigma.
\end{multline}
Since $|\log \eps|$ is bounded by $\eps^{-1}$, $\sum_{\gamma\le\alpha+\beta}\binom{\alpha+\beta}{\gamma}=2^{|\alpha+\beta|}$ and $\delta!\le|\delta|!\le |\delta|^{|\delta|}$ for all $\delta\in\N^n$ we can conclude from \eqref{conv_2_b} that
\begin{multline*}
|\partial^\alpha(u\ast\rho_\eps)(x)|\le c\,c_1^{|\alpha+\beta|}\eps^{-|\alpha|-|\beta|}\sum_{\gamma\le\alpha+\beta}\binom{\alpha+\beta}{\gamma}|\alpha+\beta-\gamma|^{\sigma|\alpha+\beta-\gamma|}|\gamma|^{\sigma|\gamma|}\\
\le c\, c_1^{|\alpha+\beta|}\eps^{-|\alpha|-|\beta|}2^{|\alpha+\beta|}|\alpha+\beta|^{\sigma|\alpha+\beta|}\\
\le c c_1^{|\alpha+\beta|}\eps^{-|\alpha|-|\beta|} 2^{|\alpha+\beta|}\espo^{\sigma|\alpha+\beta|}|\alpha+\beta|!^\sigma.
\end{multline*}
At this point collecting the terms with exponent $|\alpha|$ and the terms with exponent $|\beta|$ ($\beta$ depends only on $u$) we conclude that there exist a constants $C>0$ and $C_1>0$ such that 
\beq
\label{final_est}
|\partial^\alpha(u\ast\rho_\eps)(x)|\le C_1^{|\beta|}|\beta|!^\sigma C^{|\alpha|}|\alpha!|^\sigma \eps^{-|\alpha|-|\beta|},
\eeq
uniformly in $\eps\in(0,1/2]$. Note that by the inequality $|\delta|!\le n^{|\delta|}\delta !$ we have that \eqref{final_est} implies the assertion of Proposition \ref{prop_reg_distr} with $N=|\beta|$ and different constants.
\end{proof}
It follows that the net $(u\ast\rho_\eps)_\eps$ is $\gamma^\sigma_c$-moderate and therefore from Proposition \ref{prop_reg_gevrey_mod} we have that there exists $c>0$ and $N\in\N$ such that
\[
|\widehat{u\ast\rho_\eps}(\xi)|\le c\eps^{-N}\esp^{-c\eps^{\frac{1}{\sigma}}\lara{\xi}^{\frac{1}{\sigma}}},
\]
for all $\xi\in\R^n$ and $\eps$ small enough (from the proof, $\eps\in(0,1/2]$).

\begin{remark}
Starting from Proposition \ref{prop_reg_distr} and arguing as for the embedding of $\E'(\R^n)$ into $\G(\R^n)$ one can easily prove that
\[
\E'(\R^n)\to \G^\sigma_c(\R^n): u\mapsto [(u\ast\rho_\eps)_\eps]
\]
is an embedding of $\E'(\R^n)$ into $\G^\sigma_c(\R^n)$.
\end{remark}

\subsection{Energy estimates and well-posedness}
Let us now consider the Cauchy problem
\[
\begin{split}
D_t^2u(t,x)-\sum_{i=1}^n b_i(t)D_tD_{x_i}u(t,x)-\sum_{i=1}^n a_i(t)D_{x_i}^2u(t,x)&=0,\\
u(0,x)&=g_0,\\
D_t u(0,x)&=g_1,\\
\end{split}
\]
with $g_0,g_1\in \E'(\R^n)$. We embed coefficients and initial data in the corresponding Colombeau algebras and we transform the equation into a first order system similarly to Case 1 and 2. In, particular from Proposition \ref{prop_reg_distr} we have in this case that the initial data $V_{\eps}(0,\xi)$ fulfils
\[
|V_\eps(0,\xi)|\le \eps^{-N}C'_0\esp^{-C_0\eps^{\frac{1}{s}}\lara{\xi}^{\frac{1}{s}}},
\]
for some $N\in\N$. This modifies the estimates of Case 2 only by a multiplying factor $\eps^{-N}$. So for $R_\eps$ as in \eqref{R_eps_2} we get that there exists $N'\in\N$ such that
\[
|V_\eps(t,\xi)|\le c'\esp^{-C'\eps^{\frac{1}{s}}\lara{\xi}^{\frac{1}{s}}}\eps^{-N'},
\]
for all $\eps$, $t\in(0,T]$ and $\xi\in\R^n$. This result allows us to state the following well-posedness theorem.
\begin{theorem}
\label{theo_CP_3}
Let 
\[
\begin{split}
D_t^2u(t,x)-\sum_{i=1}^n b_i(t)D_tD_{x_i}u(t,x)-\sum_{i=1}^n a_i(t)D_{x_i}^2u(t,x)&=0,\\
u(0,x)&=g_0,\\
D_t u(0,x)&=g_1,\\
\end{split}
\]
where the coefficients $a_i$ and $b_i$ are real valued distributions with compact support contained in $[0,T]$ and $a_i$ is non-negative for all $i=1,\dots,n$. Then, the conclusion of Theorem \ref{theo_CP_2} holds for initial data $g_0$ and $g_1$ in $\E'(\R^n)$ as well.  
\end{theorem}

\section{Consistency with the classical well-posedness results}

We conclude this paper by showing that when the coefficients are regular enough and the initial data are Gevrey then the 
very weak solution coincides with the classical and ultradistributional ones obtained in \cite{GR:12, KS}.
\begin{theorem}
\label{theo_consistency}
Let 
\beq
\label{CP_cons}
\begin{split}
D_t^2u(t,x)-\sum_{i=1}^n b_i(t)D_tD_{x_i}u(t,x)-\sum_{i=1}^n a_i(t)D_{x_i}^2u(t,x)&=0,\\
u(0,x)&=g_0,\\
D_t u(0,x)&=g_1,\\
\end{split}
\eeq
where the real-valued coefficients $a_i$ and $b_i$ are compactly supported, belong to $C^k([0,T])$ with $k\ge 2$ and $a_i\ge 0$ for all $i=1,\dots,n$. Let $g_0$ and $g_1$ belong to $\gamma^s_c(\R^n)$ with $s>1$. Then 
\begin{itemize}
\item[(i)] there exists an embedding of the coefficients $a_i$'s and $b_i$'s, $i=1,\dots,n$, into $\G([0,T])$, 
such that 
the Cauchy problem above has a unique solution $u\in \G([0,T];\G^s(\R^n))$ provided that
\[
1< s<1+\frac{k}{2};
\] 
\item[(ii)] any representative $(u_\eps)_\eps$ of $u$ converges in $C([0,T];\gamma^s(\R^n))$ as $\eps\to0$ to the unique classical solution in $C^2([0,T], \gamma^s(\R^n))$ of the Cauchy problem \eqref{CP_cons};
\item[(iii)] if the initial data $g_0$ and $g_1$ belong to $\E'(\R^n)$ then any representative $(u_\eps)_\eps$ of $u$ converges in $C([0,T];\D'_{(s)}(\R^n))$ to the ultradistributional solution in $C^2([0,T],\D'_{(s)}(\R^n))$ of the Cauchy problem \eqref{CP_cons}.
\end{itemize}
\end{theorem}
\begin{proof}
(i) From Section \ref{sec_case1} (Case 1) we know that by embedding coefficients and initial data in the corresponding Colombeau algebras the Cauchy problem has a unique solution $u\in \G([0,T]; \G^s(\R^n))$. It also has a unique classical solution $\wt{u}\in C^2([0,T], \gamma^s(\R^n))$. 

(ii) We now want to compare $u$ with $\wt{u}$. By definition of classical solution we know that
\beq
\label{CP_class}
 \begin{split}
D_t^2\wt{u}(t,x)-\sum_{i=1}^n b_i(t)D_tD_{x_i}\wt{u}(t,x)-\sum_{i=1}^n a_i(t)D_{x_i}^2\wt{u}(t,x)&=0,\\
\wt{u}(0,x)&=g_0,\\
D_t \wt{u}(0,x)&=g_1.\\
\end{split}
\eeq
Since the initial data do not need to be regularised because they are already Gevrey there exists a representative $(u_\eps)_\eps$ of $u$ such that
\beq
\label{CP_class_2}
\begin{split}
D_t^2u_\eps(t,x)-\sum_{i=1}^n b_{i,\eps}(t)D_tD_{x_i}u_\eps(t,x)-\sum_{i=1}^n a_{i,\eps}(t)D_{x_i}^2u_\eps(t,x)&=0,\\
u_\eps(0,x)&=g_{0},\\
D_t u_\eps(0,x)&=g_{1},\\
\end{split}
\eeq
for suitable embeddings of the coefficients $a_i$ and $b_i$. Noting that the nets $(a_{i,\eps}-a_i)_\eps$ and $(b_{i,\eps}-b_i)_\eps$ are converging to $0$ in $C([0,T]\times\R^n)$ for $i=1,\dots,n$ we can rewrite \eqref{CP_class} as
\beq
\label{CP_class_3}
\begin{split}
D_t^2\wt{u}(t,x)-\sum_{i=1}^n b_{i,\eps}(t)D_tD_{x_i}\wt{u}(t,x)-\sum_{i=1}^n a_{i,\eps}(t)D_{x_i}^2\wt{u}(t,x)&=n_\eps(t,x),\\
\wt{u}(0,x)&=g_0,\\
D_t \wt{u}(0,x)&=g_1,\\
\end{split}
\eeq
where $n_\eps\in C([0,T];\gamma^s(\R^n))$ and converges to $0$ in this space. From \eqref{CP_class_3} and \eqref{CP_class_2} we get that $\wt{u}-u_\eps$ solves the Cauchy problem
\[
\begin{split}
D_t^2(\wt{u}-u_\eps)(t,x)-\sum_{i=1}^n b_{i,\eps}(t)D_tD_{x_i}(\wt{u}-u_\eps)(t,x)-\sum_{i=1}^n a_{i,\eps}(t)D_{x_i}^2(\wt{u}-u_\eps)(t,x)&=n_\eps(t,x),\\
(\wt{u}-u_\eps)(0,x)&=0,\\
(D_t \wt{u}-D_t u_\eps)(0,x)&=0,\\
\end{split}
\]
By the energy estimates of Case 1 and arguing as in the uniqueness proof of Theorem \ref{theo_CP_1} to deal with the right-hand side
we arrive after reduction to a system and by application of the Fourier transform to estimate $|(\wt{V}-V_\eps)(t,\xi)|$ as in \eqref{est_uniqueness}, in terms of  $(\wt{V}-V_\eps)(0,\xi)$ and the right-hand side $n_\eps(t,x)$. In particular, since the coefficients are regular enough (of class $C^k$, $k\ge 2$), the term $\omega(\eps)^{-N}$ disappears in \eqref{est_uniqueness} and we simply get
\beq
\label{impor_est_const}
|(\wt{V}-V_\eps)(t,\xi)|\le c_1\lara{\xi}^N\esp^{\kappa_1 \eps^{\frac{1}{s}}\lara{\xi}^{\frac{1}{s}}}|(\wt{V}-V_\eps)(0,\xi)|+c_2\lara{\xi}^N\esp^{\kappa_2\eps^{\frac{1}{s}}\lara{\xi}^{\frac{1}{s}}}|\widehat{n_\eps}(t,\xi)|
\eeq
Since  $(\wt{V}-V_\eps)(0,\xi)=0$ and $n_\eps\to 0$ in $C([0,T];\gamma^s(\R^n))$ we conclude that $u_\eps\to \wt{u}$ in $C([0,T];\gamma^s(\R^n))$. Moreover, since
 any other representative of $u$ will differ from $(u_\eps)_\eps$ by a $C^\infty([0,T];\gamma^s(\R^n))$-negligible net, 
 the limit is the same for any representative of $u$.

(iii) Let us now consider the case of initial data in $\E'(\R^n)$. We know from \cite{GR:12} that the Cauchy problem 
\beq
\label{CP_const_ultra}
\begin{split}
D_t^2u(t,x)-\sum_{i=1}^n b_i(t)D_tD_{x_i}u(t,x)-\sum_{i=1}^n a_i(t)D_{x_i}^2u(t,x)&=0,\\
u(0,x)&=g_0,\\
D_t u(0,x)&=g_1,\\
\end{split}
\eeq
has a unique solution $\wt{u}\in C^2([0,T],\D'_{(s)}(\R^n))$ in the sense of ultradistributions. Hence,
\[
\begin{split}
D_t^2\wt{u}(t,x)-\sum_{i=1}^n b_i(t)D_tD_{x_i}\wt{u}(t,x)-\sum_{i=1}^n a_i(t)D_{x_i}^2\wt{u}(t,x)&=0,\\
\wt{u}(0,x)&=g_0,\\
D_t \wt{u}(0,x)&=g_1.\\
\end{split}
\]
We also know that the Cauchy problem \eqref{CP_const_ultra} has a unique solution $u$ in $\G([0,T];\G^s(\R^n))$ after suitable embedding of coefficients and initial data. This means that there exists a representative $(u_\eps)_\eps$ of classical smooth solutions  such that  
\beq
\label{proof_cons}
\begin{split}
D_t^2u_\eps(t,x)-\sum_{i=1}^n b_{i,\eps}(t)D_tD_{x_i}u_\eps(t,x)-\sum_{i=1}^n a_{i,\eps}(t)D_{x_i}^2u_\eps(t,x)&=0,\\
u_\eps(0,x)&=g_{0,\eps},\\
D_t u_\eps(0,x)&=g_{1,\eps},\\
\end{split}
\eeq
for suitable embeddings of coefficients and initial data as discussed previously in Case 3. Note that the nets $(a_{i,\eps}-a_i)_\eps$ and $(b_{i,\eps}-b_i)_\eps$ are converging to $0$ in $C([0,T]\times\R^n)$ for $i=1,\dots,n$ and that $g_{0,\eps}-g_0$ and $g_{1,\eps}-g_1$ are nets of distributions converging to 0 as well. As in (ii) we can write
\beq
\label{last_CP}
 \begin{split}
D_t^2(u_\eps-\wt{u})(t,x)-\sum_{i=1}^n b_{i,\eps}(t)D_tD_{x_i}(u_\eps-\wt{u})(t,x)-\sum_{i=1}^n a_{i,\eps}(t)D_{x_i}^2(u_\eps-&\wt{u})(t,x)= n_\eps(t,x),\\
u_\eps(0,x)-\wt{u}(0,x)&=g_{0,\eps}-g_0,\\
D_t u_\eps(0,x)-D_t\wt{u}(0,x)&=g_{1,\eps}-g_1,\\
\end{split}
\eeq
where $(n_\eps)_\eps$ is converging to $0$ in $C([0,T];\D'_{(s)}(\R^n))$ and the nets $g_{0,\eps}-g_0$ and $g_{1,\eps}-g_1$ are converging to $0$ in the sense of distributions.  From the estimate \eqref{impor_est_const} we deduce that $\wt{V}-V_\eps\to 0$ in  $C([0,T];\D'_{(s)}(\R^n))$ or in other words that $u_\eps\to \wt{u}$ in $C([0,T];\D'_{(s)}(\R^n))$. Analogously, this result is not affected by changing the representative $(u_\eps)_\eps$ of $u\in \G([0,T];\G^s(\R^n))$.

\end{proof}


\end{document}

%% file: mymacro.tex
\newtheorem{theorem}{Theorem}[section]

\newtheorem{lemma}[theorem]{Lemma}
\newtheorem{proposition}[theorem]{Proposition}
\newtheorem{definition}[theorem]{Definition}
\theoremstyle{definition}
\newtheorem{remark}[theorem]{Remark}

\newcommand{\wt}[1]{\widetilde{#1}}

\newcommand{\Cinf}{\ensuremath{\mathcal{C}^\infty}}
\newcommand{\Cinfc}{\ensuremath{\mathcal{C}^\infty_{\text{c}}}}
\newcommand{\D}{\ensuremath{{\mathcal D}}}
\renewcommand{\S}{\mathscr{S}}
\newcommand{\E}{\ensuremath{{\mathcal E}}}

\newcommand{\mF}{\mathcal{F}}

\newcommand{\mb}[1]{\ensuremath{\mathbb{#1}}}
\newcommand{\N}{\mb{N}}

\newcommand{\R}{\mb{R}}
\newcommand{\C}{\mb{C}}

\newcommand{\G}{\ensuremath{{\mathcal G}}}
\newcommand{\Neg}{\mathcal{N}}

\newcommand{\lara}[1]{\langle #1 \rangle}

\newcommand{\supp}{\mathrm{supp}}

\newfont{\bigmath}{cmr12 at 13pt}

\newfont{\grecomath}{cmmi12 at 15pt}

\newcommand{\esp}{\mathrm{e}}

\newfont{\bl}{msbm10 scaled \magstep2}

\newcommand{\beq}{\begin{equation}}
\newcommand{\eeq}{\end{equation}}

\newcommand{\notmid}{\mid\kern-0.5em\not\kern0.5em}

\newcommand{\inner}[3][\empty]{\ifx#1\empty\left( #2|#3\right)\else#1( #2|#3 #1)\fi}

\newcommand{\eps}{\varepsilon}

\renewcommand{\Re}{\ensuremath{\mathrm{Re}}}
\renewcommand{\Im}{\ensuremath{\mathrm{Im}}}

